\documentclass[12pt]{amsart}%
\usepackage{amsfonts}
\usepackage{amsmath}
\usepackage{amssymb}
\usepackage{graphicx}
\usepackage{color}%
\usepackage{booktabs}
\usepackage{array}
\usepackage[pagebackref]{hyperref}
\makeatletter
\@addtoreset{equation}{section}
\makeatother

\newtheorem{theorem}{Theorem}[section]

\newtheorem{corollary}[theorem]{Corollary}

\newtheorem{definition}[theorem]{Definition}

\newtheorem{lemma}[theorem]{Lemma}

\newtheorem{proposition}[theorem]{Proposition}
\newtheorem{remark}[theorem]{Remark}

\begin{document}
\title[A complete characterization of the effect of sharp $L^p$ perturbation ]{A complete characterization of the existence of extremals for the Trudinger-Moser inequality on $\mathbb{R}^2$
 under sharp $L^p$-perturbations}
\author{Lu Chen, Rou Jiang, Guozhen Lu and Maochun Zhu}
\address{Key laboratory of Algebraic Lie Theory and Analysis of Ministry of Education, School of Mathematics and Statistics, Beijing Institute of Technology, Beijing
100081, P. R. China.}
\email{chenlu5818804@163.com}
\address{School of Mathematics and Statistics, Wuhan University, Wuhan, 430072, P. R. China.}
\email{jiangrou2023@126.com}
\address{Department of Mathematics\\
University of Connecticut\\
Storrs, CT 06269, USA}
\email{guozhen.lu@uconn.edu}
\address{School of Mathematics and Statistics, Nanjing University of Science and
Technology, Nanjing, 210094, P. R. China.}
\email{zhumaochun2006@126.com}
\thanks{The first author was partly supported by National Natural Science
Foundation of China (No. 12271027). The second author was supported by Natural
Science Foundation of China (12471056). The third author was supported by a grant from the Simons Foundation. The fourth author was supported by
Natural Science Foundation of China (12571122). }

\begin{abstract}
In this paper, we investigate the effect of sharp $L^p$ perturbations on the
existence and nonexistence of extremals for the critical Trudinger--Moser
inequality on $\mathbb R^2$:
$$
S(\lambda,p)
:=
\sup_{\substack{u\in H^{1}(\mathbb R^{2})\\
\int_{\mathbb R^2}(|\nabla u|^2+|u|^2)\,dx\le 1}}
\int_{\mathbb R^2}
\left(e^{4\pi u^2}-1-\lambda |u|^p\right)\,dx .
$$
For $2<p\le 4$, we prove the existence of a critical value
$\lambda^{\ast}\in(0,+\infty)$ such that $S(\lambda,p)$ is attained when
$\lambda<\lambda^{\ast}$ and is not attained when
$\lambda>\lambda^{\ast}$. Moreover, we show that the nonattainment in this
range is caused by a vanishing phenomenon. 

For $p=2$, combining our analysis with the nonexistence results for
$L^2$-perturbed Trudinger--Moser inequalities obtained in \cite{Chenluzhu},
we establish the existence of two finite thresholds
$\lambda_{\ast}>-\infty$ and $\lambda^{\ast}<+\infty$ such that
$S(\lambda,2)$ is attained when
$\lambda_{\ast}<\lambda<\lambda^{\ast}$, and is not attained when
$\lambda<\lambda_{\ast}$ or $\lambda>\lambda^{\ast}$.

In contrast, for $p>4$, we prove that $S(\lambda,p)$ is attained for all
admissible values of $\lambda$.

Our results indicate that, in the whole-space setting, the $L^p$-perturbation term affects the existence and nonexistence  of extremals  through either concentration or vanishing phenomena, which is fundamentally different from the bounded-domain case, where existence or nonexistence is governed solely by concentration phenomena.
 
 These results provide a complete
characterization of how sharp $L^p$ perturbations determine the existence and
nonexistence of extremals for critical Trudinger--Moser inequalities on the entire 
$\mathbb R^2$. 
The resulting existence and nonexistence theory exhibits a threshold
structure with respect to the $L^p$ pertubation reminiscent of the classical Brezis--Nirenberg phenomenon in the whole space $\mathbb R^2$. 
 
\end{abstract}
\maketitle

\section{Introduction}

Let $\Omega\subset\mathbb{R}^{N}$, $N\geq2$ be a bounded domain and
$W_{0}^{1,p}(\Omega)$ be the Sobolev space with functions vanishing on
boundary $\partial\Omega$, that is, the completion of $C_{0}^{\infty}(\Omega)$
under the norm%
\[
\Vert\nabla u\Vert_{L^{p}(\Omega)}=\left(  \int_{\Omega}|\nabla u|^{p}%
dx\right)  ^{\frac{1}{p}}\text{.}%
\]
The classical Sobolev embedding theorem asserts that $W_{0}^{1,N}\left(
\Omega\right)  \hookrightarrow L^{p}\left(  \Omega\right)  $ for any $p>1$.
However, $W_{0}^{1,N}\left(  \Omega\right)  \hookrightarrow L^{\infty}\left(
\Omega\right)  $ does not hold. One knows from the works of Yudovi\v{c}
\cite{Yud}, Poho\v{z}aev \cite{Poh} and Trudinger \cite{Tru} that $W_{0}%
^{1,N}\left(  \Omega\right)  $ can be imbedded into the Orlicz space
${L_{\phi_{N}}}(\Omega)$ with the function $\phi_{N}(t)={\exp}\left(
{{{\left\vert t\right\vert }^{\frac{N}{{N-1}}}}}\right)  -1.$ In 1971, Moser
\cite{Mos} sharpened this embedding and proved the following \textit{Trudinger-Moser}
inequality
\begin{equation}
\sup_{\Vert\nabla u\Vert_{L^{N}(\Omega)}\leq1}\int_{\Omega}e^{\alpha
|u|^{\frac{N}{N-1}}}dx\leq c\left\vert \Omega\right\vert, \text{ iff }%
\alpha\leq\alpha_{N}:=N\omega_{N-1}^{\frac{1}{N-1}}, \label{1111}%
\end{equation}
where $\omega_{N-1}$ denotes
$\left(  N-1\right)  $-dimensional surface measure of the unit ball in $\mathbb{R}^{N}$.

\medskip

When $\Omega$ is the whole space $\mathbb{R}^{N}$,  there are several extensions of the Trudinger-Moser inequality, see Cao \cite{Cao92} in the case $N =2$ and for any dimension $N \geq 2$ by do \'{O} \cite{do O97}.  A sharp  Trudinger-Moser inequality on $\mathbb{R}^{N}$ was obtained by Adachi-Tanaka \cite{Adachi-Tanaka}  in the following form
 \[
\sup_{u \in W^{1,N}(\mathbb{R}^N), \|\nabla u\|_N \le 1} \frac{1}{\|u\|_N^N} \int_{\mathbb{R}^N} \Phi_N \left( \alpha |u|^{\frac{N}{N-1}} \right) dx < \infty,
\]
for any $\alpha<\alpha_{N}$, where $\Phi_N\left(  t\right)
=e^{t}-\sum\limits_{j=0}^{N-2}\frac{t^{j}}{j!}$. Unlike in the inequality 
(\ref{1111}), the result of \cite{Adachi-Tanaka} has a subcritical form. Later, in \cite{Ruf} and \cite{liruf}, Li and Ruf 
showed that the exponent $\alpha_N$ becomes admissible if the Dirichlet norm  is 
replaced by the standard Sobolev norm $$\left\vert \left\vert
u\right\vert \right\vert _{W^{1,N}\left(  \mathbb{R}^{N}\right)  }=\left(
\int_{\mathbb{R}^{N}}\left(  \left\vert \nabla u\right\vert ^{N}+\left\vert
u\right\vert ^{N}\right)  dx\right)  ^{\frac{1}{N}},$$  more precisely, they proved that
\begin{equation}
J\left(\alpha\right):=\underset{u\in W^{1,N}\left(  \mathbb{R}^{N}\right)  ,\left\vert \left\vert
u\right\vert \right\vert _{W^{1,N}\left(  \mathbb{R}^{N}\right)  }\leq1}{\sup
}\int_{\mathbb{R}^{N}}\Phi_N\left(  \alpha\left\vert u\right\vert ^{\frac
{N}{N-1}}\right)  dx<+\infty\text{, iff }\alpha\leq\alpha_{N} .\label{2222}%
\end{equation}
All the proofs of both subcritical and critical Trudinger-Moser inequalities  use the symmetrization principle based on the P\'olya-Szeg\"{o} inequality in the Euclidean space. However, the P\'olya-Szeg\"{o} inequality fails in many non-Euclidean settings such as the Heisenberg group or on higher order Sobolev spaces even on the Euclidean spaces. To overcome this obstacle, Lam and Lu developed a symmetrization-free method in \cite{LamLu-AIM, LamLu-JDE} to establish the sharp Trudinger-Moser inequality on the Heisenberg group and Adams inequalities of any fractional order on Euclidean space. (see also \cite{LamLuTang-NA} for subcritical Trudinger-Moser inequality on the Heisenberg group). Moreover, it was proved by Lam et al. \cite{LamLuzhang-07}(see also \cite{Cassani14}) that the critical and subcritical Trudinger-Moser inequalities are equivalent and identities between the supremums for the critical and subcritical Trudinger-Moser functionals were established in \cite{LamLuzhang-07}.
 
\medskip

An interesting question related to 
Trudinger-Moser inequalities is
 whether or not their extremal functions exist? For the bounded domain, this question was first studied by Carleson and Chang \cite{Car}. They proved the existence of maximizers
 for the Trudinger-Moser inequality \eqref{1111} on the unit ball through symmetrization rearrangement inequality combining with the
ODE technique. Struwe \cite{Struwe88} subsequently extended this to domains sufficiently close to balls. The existence of maximizers on arbitrary bounded domains was then established by Flucher \cite{Flu} in dimension two and Lin \cite{Lin} in higher dimensions, both employing the method of harmonic transplantation. For further existence results on bounded domains, including those for Trudinger-Moser inequalities involving the $L^p$  norm,  or Trudinger-Moser inequalities on compact Riemannian manifolds,  one can see \cite{Adimurthi-Struwe, Yang06, Li1, Li2}, and the references therein.  It should be noted that the existence results in these works hinge crucially on the blow-up analysis pioneered by Adimurthi and Struwe \cite{Adimurthi-Struwe}, Y. X. Li \cite{Li1, Li2}, and Adimurthi and Druet \cite{Adimurthi-Druet}, which is now a 
standard method of dealing with the existence of maximizers for Trudinger-Moser type inequalities.


We also note that in \cite{Man}, Mancini and Martinazzi presented a completely
different approach for proving the existence of extremals for the
Trudinger-Moser inequality on the disk. They established the existence result
via a sharp Dirichlet energy expansion formula for sequences of subcritical
maximizers, a method based on techniques introduced in \cite{Malchiodi} and
involves performing a Taylor expansion of the subcritical maximizers near the
blow-up point. This method can also be utilized to explore the nonexistence of
extremals for Trudinger-Moser type inequalities. For example, based on the
works of \cite{Malchiodi,Man,Dru2}, Mancini and Thizy \cite{Man2} obtained a nonexistence result for the Adimurthi-Druet inequality on bounded planar domains(see \cite{Adimurthi-Druet}). Furthermore,
using a similar approach, Thizy \cite{Thi} provided some sharp conditions for
the existence and nonexistence of the following perturbed Trudinger-Moser
inequality on $\Omega\subset\mathbb{R}^{2}$:
\[
S_{g,4\pi}\left(  \Omega\right)  :=\underset{u\in H_{0}^{1}(\Omega),\left\vert
\left\vert \nabla u\right\vert \right\vert _{L^{2}\left(  \Omega\right)  }%
^{2}\leq1}{\sup}\int_{\Omega}\left(  1+g\left(  u\right)  \right)  \exp\left(
4\pi u^{2}\right)  dx,
\]
where $g$ satisfies $g\left(  t\right)  \rightarrow0$ as $t\rightarrow\infty$.
Thizy's result indicates that exponential perturbations can influence the
existence and non-existence of extremals of the Trudinger-Moser inequality on
 bounded planar domains. In 2002, de Fegueiredo, do \'{O} and Ruf \cite{FDR} considered the following $L^{2}$ perturbed maximization problem\begin{equation}S_{\Omega}(\lambda,2)=\underset{u\in H_{0}^{1}(\Omega),\left\vert \left\vert \nabla u\right\vert \right\vert _{L^{2}\left(\Omega\right)  }^{2}\leq1}{\sup}\int_{\Omega}\left(  e^{4\pi u^{2}}-1-\lambda|u|^{2}\right)  dx, \label{permu}\end{equation} where $\Omega$ is the unit ball of $\mathbb{R}^{2}$, and they established the existence of a maximizer for any $\lambda<4\pi$. They also conjectured that for $\lambda\geq 4\pi$, $S_{\Omega}(\lambda,2)$ would not be attained. Subsequently, in \cite{Li2006}, Y.X. Li gave a negative answer to this conjecture by showing that the supremum above is still attained even $\Omega$ is a general bounded domain when $\lambda$ is slightly larger than $4\pi$. Recently, Hashizume \cite{Has} revisited the perturbed maximization problem \eqref{permu} in the unit disc  and proved the existence of a positive threshold for the attainment and non-attainment of extremals.  Chen et al. \cite{CJZ-SIAM} further gave a complete characterization on attainment and non-attainment of extremals for problem \eqref{permu} on any bounded domain by developing a comparison principle and establishing a sharp Dirichlet energy expansion formula for non-radial maximizing sequences. It should be noted that the non-attainment results in \cite{Has,CJZ-SIAM} arise from the refined blow up analysis for concentration phenomenon.

\medskip
 Regarding the Trudinger-Moser inequality \eqref{2222} on whole space $\mathbb{R}^{N}$,   the existence of maximizers was proved by  Ruf  \cite{Ruf} and by Li and Ruf  \cite{liruf}.  Lu and Zhu \cite{Luzhu-09} subsequently addressed the existence of maximizers for Trudinger-Moser-type inequalities involving $L^N$
 norm in the entire space. We also refer to Ishiwata \cite{Ish} for the non-existence of maximizers of the Trudinger-Moser inequality \eqref{2222} in the subcritical case $0<\alpha<4\pi$  for sufficiently small $\alpha$. For related existence and non-existence results in the subcritical cases, see the work of Lam et al  \cite{LamLuzhang-19} where the authors apply the identities between the supremums for the critical and subcritcal Trudinger-Moser inequalities \cite{LamLuzhang-19},  and the work by Ikoma et al. \cite{Ikoma19}. 
 
 In the critical case, there are also some existence and non-existence results. In the recent work
\cite{Chenluzhu}, Chen et al. \cite{Chenluzhu} consider the following perturbed maximization problem on $\mathbb{R}^{2}$: 
 \begin{equation}\label{advnon}
S(\lambda,2)=\sup_{\substack{u\in H^{1}\left(  \mathbb{R}^{2}\right)\\
\int_{\mathbb{R}^2}\left(|\nabla u|^2 + |u|^2\right)dx \le 1}}
\int_{\mathbb{R}^{2}}\left(  e^{4\pi u^{2}}-1-\lambda |u|^{2}\right)  dx,\end{equation}
when $\lambda<4\pi$. They also showed the existences of  finite threshold $\lambda_*<4\pi-8\pi^2 B_2<0$, here $B_2$ the following Gagliardo--Nirenberg--Sobolev best constant 
\[
B_{2}=\underset{u\in H^{1}\left(  \mathbb{R}^{2}\right)  ,u\neq0}{\sup}%
\frac{\left\vert \left\vert u\right\vert \right\vert _{4}^{4}}{\left\vert
\left\vert u\right\vert \right\vert _{2}^{2}\left\vert \left\vert \nabla
u\right\vert \right\vert _{2}^{2}}%
\]
which is known to be larger than $\frac{1}{2\pi}$(see \cite{Weinstein}),
such that $S(\lambda,2)$ is attained for $\lambda_*<\lambda<4\pi$, and is not attained for  $\lambda<\lambda_*$. In the latter case, 
$S\left(  \lambda,2\right)  =4\pi-\lambda$, which is the optimal vanishing level of Trudinger-Moser functional on
$\mathbb{R}^{2}$. These findings reveal that
the vanishing phenomenon (first discovered in \cite{Ish}) on the entire space
can influence the nonexistence of extremals  when $\lambda$ is sufficiently negative.  For the attainability results of the supremum \eqref{advnon} in higher dimensions, we refer to the work of Nguyen \cite{Nguyen-JFA}. It is also worth mentioning that Nguyen \cite{Nguyen21} obtained an existence and non-existence result in the critical case with inhomogeneous constraints: $M(a,N)=\{u\in W^{1,N}(\mathbb{R}^N): \| \nabla u \|_N^a + \| u \|_N^N = 1 \}$:
 \[
S(a) = \sup_{ u\in M(a,N)} \int_{\mathbb{R}^N} \Phi_N(\alpha_N |u|^{\frac{N}{N-1}})dx.
\]
In that work, the author proves the existences of thresholds $a^*$ and $a_*$ such that $S(a)$  is attained for any $a\in (a_*,a^*)$,  while it is not attained for $a<a_*$ or $a>a^*$, by exploting  a special transformation of functions
between the classes $M(a,N)$. We remark that  when  $a>a^*$, the non-attainment is due to the concentration phenomenon. 
\medskip

Revisiting the perturbed maximization problem \eqref{advnon}, one naturally asks what happens in the regime $\lambda\geq 4\pi$? In particular, do extremals of the supremum \eqref{advnon} still exist?  Does the concentration phenomenon on the whole space also affect the existence and non-existence of extremals for $S(\lambda,2)$, as was observed in \cite{Nguyen21} and \cite{CJZ-SIAM} on bounded domains? In this paper, we are devoted to exploring this question by focusing on the
following general $L^{p}$ perturbed maximization problem:
\[
S(\lambda,p)=\sup_{\substack{u\in H^{1}\left(  \mathbb{R}^{2}\right)\\
\int_{\mathbb{R}^2}\left(|\nabla u|^2 + |u|^2\right)dx \le 1}}\int
_{\mathbb{R}^{2}}\left(  e^{4\pi u^{2}}-1-\lambda |u|^{p}\right)  dx,
\]
when $\lambda\in\mathbb{R}$. We will give a complete
characterization of how $L^{p}$-type perturbations affect the existence of
extremal functions for critical Trudinger-Moser inequalities in $\mathbb{R}%
^{2}$. Our main results can be stated as follows:
\begin{theorem}
\label{th1.2}For $2<p\leq4$, there exists a positive threshold ${\lambda}^{\ast}<+\infty$
such that $S\left(\lambda,p\right)  $ is attained for $\lambda<\lambda
^{\ast}$ and not attained for $\lambda>{\lambda}^{\ast}$, in the latter case,  $S\left(\lambda,p\right)=4\pi  $;   for any $p>4$,
$S\left(  \lambda,p\right)  $ is always attained for all $\lambda<+\infty$. 
\end{theorem}

\begin{theorem}
\label{th1.1} For $p=2$, there exists a threshold ${\lambda}^{\ast}<+\infty$ such that
$S\left(  \lambda,2\right) $ is attained when $4\pi\leq \lambda<{\lambda}^{\ast}$, and is not attained 
when $\lambda>{\lambda}^{\ast}$; in the latter case, $S\left(  \lambda,2\right)  =\pi e$, which is the
optimal concentration level of the Trudinger-Moser functional on
$\mathbb{R}^{2}$.
\end{theorem}

Combining Theorem \ref{th1.1} with the existence and non-existence result obtained in \cite{Chenluzhu}, we give a  complete characterization of how $L^2$-type perturbations can determine the existence and nonexistence of extremals for critical Trudinger-Moser inequalities in $\mathbb{R}^2$.

\begin{corollary}
There exists two thresholds ${\lambda}^{\ast}<+\infty$ and  ${\lambda}_{\ast}>-\infty$ such that $S\left(  \lambda,2\right) $ is attained when $\lambda_{\ast}<\lambda<{\lambda}^{\ast}$, is not attained when $\lambda<{\lambda}_{\ast}$ or $\lambda>{\lambda}^{\ast}$. 
Furthermore, $S\left(  \lambda,2\right)=4\pi-\lambda$ when $\lambda<{\lambda}_{\ast}$, $S\left(  \lambda,2\right)=\pi e$ when $\lambda>{\lambda}_{\ast}$.
\end{corollary}

\begin{remark}
    In the case $p=2$, the non-attainment in Theorem \ref{th1.1} stems from the concentration phenomenon, as observed in \cite{Nguyen21}. It remains unknown whether the threshold $a^*$ in \cite{Nguyen21} is finite; by contrast, we are able to prove the finiteness of $\lambda^*$ in our setting. To achieve this, we seek a contradiction by employing the refined blow-up analysis to derive a sharp Dirichlet energy expansion formula for the maximizers $u_\lambda$(if they always exist when $\lambda^*=+\infty$). We believe that the refined blow-up analysis can also be applied to determine the finiteness of the threshold $a^*$ in \cite{Nguyen21} in the two-dimensional case, which will be investigated in our subsequent work. Somewhat surprisingly, in contrast to the case $p=2$, the non-attainment in Theorem \ref{th1.2} for $2<p\leq4$ arises from the vanishing phenomenon  rather than the concentration phenomenon.
\end{remark}
Our results indicate that, in the whole-space setting, the $L^p$-perturbation term affects the existence and nonexistence of extremals through either concentration or vanishing phenomena—a feature that is fundamentally different from the bounded-domain case, where existence or nonexistence is governed solely by concentration phenomena(see \cite{CJZ-SIAM}). The difference in the effect of the $L^p$
-perturbation on existence and nonexistence between the whole space and bounded domains is summarized in the following table:

\begin{table}[h]
\tiny
\begin{tabular}{@{} p{1.7 cm} p{6.6cm} p{6.6cm} @{}}
\toprule

\textbf{Underlying domains} & \textbf{On the entire Space \(\mathbb{R}^2\)} & \textbf{On Bounded Domain \(\Omega\subset\mathbb{R}^2\)} \\
\midrule
Constraint & \(\displaystyle\int_{\mathbb{R}^2}(|\nabla u|^2+|u|^2)\,dx \le 1\) & \(\displaystyle\int_\Omega|\nabla u|^2\,dx \le 1\) \\
\midrule
Range of \(p\) & \(p\geq 2\)  & \(p\ge 1\) \\
\midrule
\rule{0pt}{2.8ex}\(p=2\) & Two thresholds \(\lambda_*<\lambda^*\): attained for \(\lambda\in(\lambda_*,\lambda^*)\); Not attained for \(\lambda<\lambda_*\) (vanishing), and for \(\lambda>\lambda^*\) (concentration).  &   A single positive threshold \(\lambda^*\):
Attained for \(\lambda<\lambda^*\); not attained for \(\lambda>\lambda^*\) (concentration). \\
\midrule
\rule{0pt}{2.8ex}\(2<p\le 4\) & One positive threshold \(\lambda^*\): attained for \(\lambda<\lambda^*\);
not attained for \(\lambda>\lambda^*\) (vanishing). & Always attained for every \(\lambda>0\) (no threshold). \\
\midrule
\rule{0pt}{2.8ex}\(p>4\) & Always attained for all \(\lambda\in\mathbb{R}\) (no threshold). & Always attained for every \(\lambda>0\) (no threshold). \\
\midrule
\rule{0pt}{2.8ex}\(1\le p<2\) & no meaning & One positive threshold \(\lambda^*>0\): attained for \(\lambda<\lambda^*\);
not attained for \(\lambda>\lambda^*\) (concentration). \\
\bottomrule
\end{tabular}
\vspace{0.5em} 
{\tiny \textbf{Table 1.1:} Comparison of the effect of $L^p$ perturbation between the whole space and bounded domains.}
\end{table}

The organization of this paper is as follows. In Section 2, we present several preliminary results concerning rearrangements, normalized concentrating and vanishing sequences, as well as the optimal concentration and vanishing levels of the perturbed Trudinger-Moser functional. Moreover, we establish some properties of $S(\lambda,p)$ that will be used later. In Section 3, we investigate the existence and non-existence of extremals for $S(\lambda,p)$ in the case $p>2$ and give the proof of Theorem \ref{th1.2}. For $p>4$, we prove the existence of extremals by  a test function argument, while for $2<p\leq4$, we establish a threshold for existence and non-existence via the vanishing phenomenon. Furthermore, we demonstrate that $S(\lambda,2)$ is attained when $\lambda$ is slightly larger than $4\pi$. In Section 4, we establish the existence of a threshold ${\lambda}^{\ast} \in (4\pi, +\infty)$  governing the existence and non-existence of extremals for $S(\lambda,2)$. The proof of Theorem \ref{th1.1} is completed by deriving sharp estimates, as $k\rightarrow\infty$, for both the Lebesgue integral term and the integral of the exponential term in the Dirichlet energy expansion of subcritical maximizers. Finally, the asymptotic behavior of sequences of subcritical maximizers is studied in Section 5 to analyze the Lebesgue integral term, while refined blow-up analysis is employed in Section 6 to estimate the integral of the exponential term.

Throughout the paper, $C$ denotes a nonnegative general constant which may vary from line to line.

\section{Preliminary results}
In this section, we present some preliminary results that will be used later. First, we recall some facts about the rearrangement.

Let $u:\mathbb{R}^{2}\rightarrow\mathbb{R}$ be a measurable function. For
$s>0$, we define the distribution function
\[
\mu_{u}\left(  s\right)  =\left\vert \{x\in\mathbb{R}^{2}:\left\vert u\left(
x\right)  \right\vert >s\}\right\vert \text{.}%
\]
Based on $\mu_{u}$, we introduce the non-increasing rearrangement $u^{\#}:[0,\infty
)\rightarrow\lbrack0,\infty)$ as
\[
u^{\#}\left(  t\right)  =\inf\{s>0:\mu_{u}\left(  s\right)  \leq t\}\text{.}%
\]
The spherically symmetric decreasing rearrangement of $u$ is defined as
$u^{\ast}\left(  x\right)  =u^{\#}\left(  \pi\left\vert x\right\vert
^{2}\right)  $. 

A fundamental property of this rearrangement is the following: let $F:\mathbb{R}\rightarrow\mathbb{R}$ be Borel measurable and assume either $F\geq0$
or $F\left(  u\right)  \in L^{1}\left(  \mathbb{R}^{2}\right)  $, then%
\begin{equation}
\int_{\mathbb{R}^{2}}F\left(  u^{\ast}\left(  x\right)  \right)
dx=\int_{\mathbb{R}^{2}}F\left(  u\left(  x\right)  \right)  dx\text{.}\label{2.1}
\end{equation}
In particular, all $L^{p}$ norms are preserved:
\begin{equation}
\int_{\mathbb{R}^{2}}\left\vert u^{\ast}\right\vert ^{p}dx=\int_{\mathbb{R}^{2}%
}\left\vert u\right\vert ^{p}dx\text{, \ \ \ }1\leq p\leq\infty\text{.}\label{2.2}
\end{equation}
Furthermore, if $u\in H^{1}\left(  \mathbb{R}^{2}\right)  $, then
$u^{\ast}\in H^{1}\left(  \mathbb{R}^{2}\right)  $ and the following
P\'olya--Szeg\H{o} inequality holds:%
\begin{equation}
    \int_{\mathbb{R}^{2}}\left\vert \nabla u^{\ast}\right\vert ^{2}dx\leq
\int_{\mathbb{R}^{2}}\left\vert \nabla u\right\vert ^{2}dx\text{.}\label{2.3}
\end{equation}

Based on the spherically symmetric decreasing rearrangement, in the following we consider $H_{rad}^{1}\left(  \mathbb{R}^{2}\right) $, the radial subspace of $H^{1}\left(  \mathbb{R}^{2}\right) $, i.e. the set of all spherically symmetric functions in 
$H^{1}\left(  \mathbb{R}^{2}\right) $. Let  \[
S_{rad}(\lambda,p)=\sup_{\substack{u\in H_{rad}^{1}\left(  \mathbb{R}^{2}\right)\\
\int_{\mathbb{R}^2}\left(|\nabla u|^2 + |u|^2\right)dx \le 1}}\int
_{\mathbb{R}^{2}}\left(  e^{4\pi u^{2}}-1-\lambda |u|^{p}\right)  dx,
\]
From  (\ref{2.1}), (\ref{2.2}) and (\ref{2.3}), we easily see that  $S\left(  \lambda,p\right)  =S_{rad}\left(  \lambda,p\right)  $, and the
existence of a maximizer of $S\left(  \lambda,p\right)  $ is equivalent to the
one of $S_{rad}\left(  \lambda,p\right)  $.

\begin{definition}
Let $\{u_{k}\}\subset H^{1}\left(  \mathbb{R}^{2}\right)  $ be a sequence such
that $u_{k}\rightharpoonup0$ in $H^{1}\left(  \mathbb{R}^{2}\right)  $.
\begin{itemize}
\item $\{u_{k}\}$ is said to be a normalized concentrating sequence (NCS) if
$\{u_{k}\}$ satisfies
\[
\left\vert \left\vert u_{k}\right\vert \right\vert _{H^{1}\left(
\mathbb{R}^{2}\right)  }=1 \quad\text{ and}\quad\underset{k\rightarrow\infty
}{\lim}\int_{B_{\rho}^{c}}\left(  \left\vert \nabla u_{k}\right\vert
^{2}+|u_{k}|^{2}\right)  dx=0\text{ for any }\rho>0.
\]

\item $\{u_{k}\}$ is said to be a normalized vanishing sequence (NVS) if
$\{u_{k}\}$ satisfies
\[
\left\vert \left\vert u_{k}\right\vert \right\vert _{H^{1}\left(
\mathbb{R}^{2}\right)  }=1\quad\text{ and}\quad\underset{k\rightarrow\infty}{\lim
}\int_{\mathbb{R}^{2}}\left\vert \nabla u_{k}\right\vert ^{2}dx=0.
\]
\end{itemize}
\end{definition}
In order to establish the existence of extremals for $S\left(  \lambda
,p\right)  $, one need to prove
\[
S\left(  \lambda,p\right)  >\max\{d_{ncl},d_{nvl}\}\text{,}%
\]
where $d_{ncl}$ is the optimal concentration level of the perturbed Trudinger-Moser functional, that is
\[
d_{ncl}=\underset{\{u_{k}\}_{k}\text{ is NCS}}{\sup}\underset{k\rightarrow
\infty}{\lim\sup}\int_{\mathbb{R}^{2}}\left(  e^{4\pi u_{k}^{2}}%
-1-\lambda\left\vert u_{k}\right\vert ^{p}\right)  dx\text{,}%
\]
and $d_{nvl}$ is its optimal vanishing level, that is
\[
d_{nvl}=\underset{\{u_{k}\}_{k}\text{ is NVS}}{\sup}\underset{k\rightarrow
\infty}{\lim\sup}\int_{\mathbb{R}^{2}}\left(  e^{4\pi u_{k}^{2}}%
-1-\lambda\left\vert u_{k}\right\vert ^{p}\right)  dx\text{.}%
\]

We first claim that
\begin{lemma}\label{CN}
For any $p\geq2$,  if $\{u_{k}\}$ is a normalized concentrating sequence, then 
\[\underset{k\rightarrow\infty}{\lim\sup}\int_{\mathbb{R}^{2}}\left(  e^{4\pi
u_{k}^{2}}-1-\lambda|u_{k}|^{p}\right)  dx\leq e\pi\text{.}%
\] Conversely, for any normalized vanishing sequence $\{u_{k}\}$, 
\[
\underset{k\rightarrow\infty}{\lim\sup}\int_{\mathbb{R}^{2}}\left(  e^{4\pi
u_{k}^{2}}-1-\lambda|u_{k}|^{2}\right) dx=4\pi-\lambda%
\]
while for $ p>2$, $$\underset{k\rightarrow\infty}{\lim\sup}\int_{\mathbb{R}^{2}}\left(  e^{4\pi
u_{k}^{2}}-1-\lambda|u_{k}|^{p}\right) dx=4\pi.$$
\end{lemma}
In order to prove Lemma \ref{CN}, we need the following two lemmas for NCS and NVS.

\begin{lemma}\label{cc}
If $\{u_{k}\}$ is a normalized concentrating sequence, then for any
$p\geq2$, we have%
\[
\underset{k\rightarrow\infty}{\lim}\int_{\mathbb{R}^{2}}\left\vert
u_{k}\right\vert ^{p}dx=0\text{.}%
\]

\end{lemma}

\begin{proof}
Since $\{u_k\}$ is an NCS, for every $\rho>0$,
\[
\int_{B_\rho^c}\bigl(|\nabla u_k|^2+|u_k|^2\bigr)\,dx\to 0.
\]
In particular,
\[
\int_{B_\rho^c}|u_k|^2\,dx\to 0.
\]

First consider $p=2$. For any fixed $\rho>0$,
\[
\int_{\mathbb R^2}|u_k|^2\,dx
=
\int_{B_\rho}|u_k|^2\,dx
+
\int_{B_\rho^c}|u_k|^2\,dx .
\]
The second term tends to $0$. For the first term, by H\"older's inequality and the Sobolev embedding
$H^1(\mathbb R^2)\hookrightarrow L^q(\mathbb R^2)$ for every $q\in [2,\infty)$,
we have, for any $q>2$,
\[
\int_{B_\rho}|u_k|^2\,dx
\leq
|B_\rho|^{1-\frac{2}{q}}
\left(\int_{B_\rho}|u_k|^q\,dx\right)^{2/q}
\leq
C_q |B_\rho|^{1-\frac{2}{q}},
\]
because $\|u_k\|_{H^1(\mathbb R^2)}=1$.

Thus
\[
\limsup_{k\to\infty}\int_{\mathbb R^2}|u_k|^2\,dx
\leq
C_q |B_\rho|^{1-\frac{2}{q}}.
\]
Letting $\rho\to0$, we obtain
\[
\int_{\mathbb R^2}|u_k|^2\,dx\to0.
\]

Now let $p>2$. Choose $q>p$. By interpolation,
\[
\|u_k\|_{L^p}
\leq
\|u_k\|_{L^2}^{\theta}
\|u_k\|_{L^q}^{1-\theta},
\]
where
\[
\frac1p=\frac{\theta}{2}+\frac{1-\theta}{q},
\qquad 0<\theta<1.
\]
Again $\|u_k\|_{L^q}$ is uniformly bounded by the Sobolev embedding, while
$\|u_k\|_{L^2}\to0$. Hence
\[
\|u_k\|_{L^p}\to0.
\]
Therefore
\[
\int_{\mathbb R^2}|u_k|^p\,dx\to0
\]
for every $p\ge2$.
\end{proof}
\begin{lemma}\label{vn}
If $\{u_{k}\}$ is a normalized vanishing sequence, then for any
$p>2$, we have%
\[
\underset{k\rightarrow\infty}{\lim}\int_{\mathbb{R}^{2}}\left\vert
u_{k}\right\vert ^{p}dx=0\text{.}%
\]

\end{lemma}

\begin{proof}

Let $p>2$. By the Gagliardo--Nirenberg inequality in dimension $2$,
\[
\|u_k\|_{L^p(\mathbb R^2)}
\le
C_p
\|\nabla u_k\|_{L^2(\mathbb R^2)}^{1-\frac{2}{p}}
\|u_k\|_{L^2(\mathbb R^2)}^{\frac{2}{p}} .
\]
Since
\[
\|\nabla u_k\|_{L^2(\mathbb R^2)}\to0
\]
and $\|u_k\|_{L^2(\mathbb R^2)}$ is bounded, it follows that
\[
\|u_k\|_{L^p(\mathbb R^2)}\to0.
\]
 
This completes the proof.
\end{proof}

Now, we give the 
\begin{proof}[Proof of Lemma \ref{CN}]
    
 For any normalized concentrating sequence $\{u_{k}\}$, Lemma \ref{cc} and a result in \cite{Ruf} directly gives that 
$$\underset{k\rightarrow\infty}{\lim\sup}\int_{\mathbb{R}^{2}}\left(  e^{4\pi
u_{k}^{2}}-1-\lambda|u_{k}|^{p}\right)  dx=\underset{k\rightarrow\infty}{\lim\sup}\int_{\mathbb{R}^{2}}\left(  e^{4\pi
u_{k}^{2}}-1\right)\leq\pi e.$$

Next assume that $\{u_k\}$ is a normalized vanishing sequence. Then
\[
\|u_k\|_{H^1(\mathbb R^2)}^2
=
\int_{\mathbb R^2}|\nabla u_k|^2\,dx
+
\int_{\mathbb R^2}|u_k|^2\,dx
=1
\]
and
\[
\int_{\mathbb R^2}|\nabla u_k|^2\,dx\to0.
\]
Hence
\[
\int_{\mathbb R^2}|u_k|^2\,dx\to1.
\]

We claim that
\[
\int_{\mathbb R^2}
\left(e^{4\pi u_k^2}-1\right)\,dx
\to 4\pi .
\]
Using the exponential expansion, we write
\[
\int_{\mathbb R^2}
\left(e^{4\pi u_k^2}-1\right)\,dx
=
4\pi\int_{\mathbb R^2}|u_k|^2\,dx
+
\sum_{j=2}^{\infty}
\frac{(4\pi)^j}{j!}
\int_{\mathbb R^2}|u_k|^{2j}\,dx .
\]
It remains to show that the higher-order terms vanish.

We use the two-dimensional Gagliardo--Nirenberg inequality with explicit
growth in the exponent (see \cite{Adachi-Tanaka}):
\[
\|u\|_{L^q(\mathbb R^2)}
\le
C\sqrt q\,
\|\nabla u\|_{L^2(\mathbb R^2)}^{1-\frac2q}
\|u\|_{L^2(\mathbb R^2)}^{\frac2q},
\qquad q\ge2.
\]
Taking $q=2j$, we get
\[
\|u_k\|_{L^{2j}}^{2j}
\le
(C\sqrt{2j})^{2j}
\|\nabla u_k\|_{L^2}^{2j-2}
\|u_k\|_{L^2}^{2}.
\]
Therefore
\[
\sum_{j=2}^{\infty}
\frac{(4\pi)^j}{j!}
\int_{\mathbb R^2}|u_k|^{2j}\,dx
\le
\|u_k\|_{L^2}^{2}
\sum_{j=2}^{\infty}
\frac{(4\pi)^j}{j!}
(C\sqrt{2j})^{2j}
\|\nabla u_k\|_{L^2}^{2j-2}.
\]
By Stirling's estimate, there exists a constant $C_0>0$ such that
\[
\frac{(4\pi)^j}{j!}(C\sqrt{2j})^{2j}
\le C_0^j
\qquad\text{for all } j\ge2.
\]
Thus
\[
\sum_{j=2}^{\infty}
\frac{(4\pi)^j}{j!}
\int_{\mathbb R^2}|u_k|^{2j}\,dx
\le
\|u_k\|_{L^2}^{2}
\sum_{j=2}^{\infty}
C_0^j
\|\nabla u_k\|_{L^2}^{2j-2}.
\]
Set
\[
a_k=\|\nabla u_k\|_{L^2}^2.
\]
Since $a_k\to0$ and $\|u_k\|_{L^2}$ is bounded,
\[
\sum_{j=2}^{\infty}C_0^j a_k^{j-1}
=
C_0^2 a_k
\sum_{m=0}^{\infty}(C_0a_k)^m
=
O(a_k).
\]
Therefore
\begin{equation}\label{5.2}
\sum_{j=2}^{\infty}
\frac{(4\pi)^j}{j!}
\int_{\mathbb R^2}|u_k|^{2j}\,dx
\to0.\end{equation}

Consequently,
\[
\int_{\mathbb R^2}
\left(e^{4\pi u_k^2}-1\right)\,dx
=
4\pi\int_{\mathbb R^2}|u_k|^2\,dx+o(1)
\to4\pi .
\]

If $p=2$, then
\[
\int_{\mathbb R^2}|u_k|^2\,dx\to1,
\]
and hence
\[
\lim_{k\to\infty}
\int_{\mathbb R^2}
\left(e^{4\pi u_k^2}-1-\lambda |u_k|^2\right)\,dx
=
4\pi-\lambda .
\]

If $p>2$, then by Lemma~\ref{vn},
\[
\int_{\mathbb R^2}|u_k|^p\,dx\to0.
\]
Therefore
\[
\lim_{k\to\infty}
\int_{\mathbb R^2}
\left(e^{4\pi u_k^2}-1-\lambda |u_k|^p\right)\,dx
=
4\pi .
\]
The proof is complete.
\end{proof}

 Below we present a criterion for the existence and non-existence of extremal functions for the supremum $S\left(  \lambda,p\right)  $. To this end, we first give some properties of the supremum $S\left(  \lambda,p\right)  $ and its extremal functions.

\begin{lemma}
\label{lem2.4}The supremum $S\left(  \lambda,p\right)  $ is non-increasing with
respect to $\lambda$ and satisfies
\[
\underset{\lambda\rightarrow
\lambda_0^{+}}{\lim}S\left(  \lambda,p\right)  =S\left(
\lambda_0,p\right)  \text{,}\newline%
\]
provided $S\left(  \lambda_0,p\right)  $ is attained. Moreover, for any $\lambda\in\mathbb{R} $, we have $S\left(  \lambda,p\right)  \geq e\pi$ when $p=2$, and $S\left(  \lambda,p\right)  \geq 4\pi$ when $p>2$.
\end{lemma}

\begin{proof}
We choose $\lambda_{1}$, $\lambda_{2}>0$ such that $\lambda_{1}>\lambda_{2}$,
and assume that $\{u_{k}\}_{k}\subset H^{1}\left(  \mathbb{R}^{2}\right)  $ is
the maximizing sequence of $S\left(  \lambda_{1},p\right)  $, then%
\begin{align*}
S\left(  \lambda_{1},p\right)   &  =\underset{k\rightarrow\infty}{\lim}%
\int_{\mathbb{R}^{2}}\left(  e^{4\pi u_{k}^{2}}-1-\lambda_{1}|u_k|^{p}\right)
dx\\
&  \leq\underset{k\rightarrow\infty}{\lim}\int_{\mathbb{R}^{2}}\left(  e^{4\pi
u_{k}^{2}}-1-\lambda_{2}|u_k|^{p}\right)  dx\leq S\left(  \lambda_{2},p\right)
\text{,}%
\end{align*}
then $S\left(  \lambda,p\right)  $ is non-increasing with respect to $\lambda$.

For $p>2$, by Lemma \ref{CN} and the definition of $S\left(  \lambda,p\right)  $, for any $\lambda>0$, it is
clear that
\begin{align*}
S\left(  \lambda,p\right)   &  =\sup_{\substack{u\in H^{1}\left(  \mathbb{R}^{2}\right)\\
\int_{\mathbb{R}^2}\left(|\nabla u|^2 + |u|^2\right)dx \le 1}}\int_{\mathbb{R}^{2}}\left(  e^{4\pi u^{2}}-1-\lambda |u|^{p}\right)
dx\\
&  \geq\underset{\{u_{k}\}_{k}\text{ is NVS}}{\sup}\underset{k\rightarrow
\infty}{\lim\sup}\int_{\mathbb{R}^{2}}\left(  e^{4\pi u_{k}^{2}}-1-\lambda |u_k|^{p}\right)  dx
 =4\pi\text{.}%
\end{align*}
While for $p=2$, we have
\begin{align}
S\left(  \lambda,2\right)   &  =\sup_{\substack{u\in H^{1}\left(  \mathbb{R}^{2}\right)\nonumber\\
\int_{\mathbb{R}^2}\left(|\nabla u|^2 + |u|^2\right)dx \le 1}}\int_{\mathbb{R}^{2}}\left(  e^{4\pi u^{2}}-1-\lambda |u|^{2}\right)
dx\nonumber\\
&  \geq\underset{\{u_{k}\}_{k}\text{ is NCS}}{\sup}\underset{k\rightarrow
\infty}{\lim\sup}\int_{\mathbb{R}^{2}}\left(  e^{4\pi u_{k}^{2}}-1-\lambda |u_k|^{2}\right)  dx\nonumber\\
&=\underset{\{u_{k}\}_{k}\text{ is NCS}}{\sup}\underset{k\rightarrow
\infty}{\lim\sup}\int_{\mathbb{R}^{2}}\left(  e^{4\pi u_{k}^{2}}-1\right)  dx=e\pi.\label{add762}
\end{align}
Let $u_{0}\in H^{1}\left(  \mathbb{R}^{2}\right)  $ be an extremal function
of $S\left(  \lambda_0,p\right)  $. Then%
\begin{align}
\underset{\lambda\rightarrow \lambda_0^{+}}{\lim}S\left(  \lambda,p\right)   &
\geq\underset{\lambda\rightarrow\lambda_0^{+}}{\lim}\int_{\mathbb{R}^{2}}\left(
e^{4\pi u_{0}^{2}}-1-\lambda |u_0|^{p}\right)  dx\nonumber\\
&  =\int_{\mathbb{R}^{2}}\left(  e^{4\pi u_{0}^{2}}-1-\lambda_0 |u_0|^{p}\right)  dx+\underset{\lambda\rightarrow\lambda_0^{+}}{\lim}(\lambda_0-\lambda)\int_{\mathbb{R}^{2}}|u_0|^{p}dx\\&=S\left(
\lambda_0,p\right)  \text{.}\label{2}%
\end{align}
On the other hand, since $S\left(  \lambda,p\right)  \leq S\left( \lambda_0,p\right)  $
for $\lambda>\lambda_0$, then
\[
\underset{\lambda\rightarrow0^{+}}{\lim}S\left(  \lambda,p\right)  \leq S\left(\lambda_
0,p\right)  \text{.}%
\]
This together with (\ref{2}) yields %
\[
\underset{\lambda\rightarrow\lambda_0^{+}}{\lim}S\left(  \lambda,p\right)  =S\left(
\lambda_0,p\right)  \text{.}%
\]
Thus, we complete the proof of Lemma \ref{lem2.4}.
\end{proof}

It is worth emphasizing that since the integrand in  $S(\lambda,p)$ is not necessarily non-negative, the norm of the extremal function could possibly be strictly less than 1. Fortunately, the following result rules out this possibility.

\begin{proposition}
\label{114}Let $t\in(0,1)$ and $p\geq 2$. Then%
\[
\sup_{\substack{u\in H^{1}\left(  \mathbb{R}^{2}\right)\\
\int_{\mathbb{R}^2}\left(|\nabla u|^2 + |u|^2\right)dx \le t}}\int_{\mathbb{R}^{2}}\left(
e^{4\pi u^{2}}-1-\lambda |u|^{p}\right)  dx\le tS\left(  \lambda,p\right)< S\left(  \lambda,p\right) \text{.}%
\]

Consequently, if $S(\lambda,p)$ is attained by some
$u_\lambda\in H^1(\mathbb R^2)$, then
\[
\|u_\lambda\|_{H^1(\mathbb R^2)}=1.
\]

\end{proposition}
\begin{proof} We prove it by a scaling argument. 
For $u\in H^1(\mathbb{R}^2)\setminus\{0\} $ with $\|u\|_{H^1(\mathbb R^2)}^2 = s \le t<1$, set

$$A = \int_{\mathbb{R}^2}|\nabla u|^2dx, B = \int_{\mathbb{R}^2}|u|^2dx,$$
so that $s = A+B$. Since $s \le t < 1$, we have $A \le s < 1$ and $B = s-A > 0.$
Define the dilated function $v(x) = u(\alpha x)$ with $\alpha = \sqrt{\frac{B}{1-A}} > 0.$ Then
$$
\int_{\mathbb{R}^2}|\nabla v|^2dx = \int_{\mathbb{R}^2}|\nabla u|^2dx = A,$$

$$\int_{\mathbb{R}^2}|v|^2dx = \alpha^{-2}\int_{\mathbb{R}^2}|u|^2dx = \frac{1-A}{B}\cdot B = 1-A,$$
so $\|v\|_{H^1(\mathbb R^2)}^2 = A + (1-A) = 1$.
For the functional

$$F(u) = \int_{\mathbb{R}^2}\Bigl(e^{4\pi u^2}-1-\lambda|u|^p\Bigr)dx,$$
a change of variables gives

$$F(v) = \int_{\mathbb{R}^2}\Bigl(e^{4\pi u(\alpha x)^2}-1-\lambda|u(\alpha x)|^p\Bigr)dx
= \alpha^{-2}F(u).$$
Hence $F(u) = \alpha^2 F(v) = \frac{B}{1-A}F(v) = \frac{s-A}{1-A}F(v).$ Because $0 \le A \le s \le t < 1$, we have

$$\frac{s-A}{1-A} \le s \le t.$$

Taking the supremum over all admissible $u$ proves
\[
\sup_{\|u\|_{H^1}^2\le t}
\int_{\mathbb R^2}
\left(
e^{4\pi u^2}-1-\lambda|u|^2
\right)\,dx
\le t
S(\lambda,p)<S(\lambda,p).
\]

Finally, suppose that $u_\lambda$ is an extremal for $S(\lambda,p)$ and
\[
\|u_\lambda\|_{H^1(\mathbb R^2)}<1.
\]
Choose $t$ such that
\[
\|u_\lambda\|_{H^1(\mathbb R^2)}^2<t<1.
\]
Then $u_\lambda$ belongs to the admissible class in the proposition, and hence
\[
S(\lambda,p)
=
\int_{\mathbb R^2}
\left(
e^{4\pi u_\lambda^2}-1-\lambda|u_\lambda|^p
\right)\,dx
\le 
tS(\lambda,p)<S(\lambda,p),\]
which is impossible. Therefore,
\[
\|u_\lambda\|_{H^1(\mathbb R^2)}=1.
\]
This completes the proof.
\end{proof}

\begin{lemma}
\label{57}Let $\lambda\geq(4-e)\pi.$ If $S\left(  \lambda,2\right)  >e\pi$, then $S\left(  \lambda,2\right)
$ is attained. Conversely, if there exists some
$\lambda^{\ast}>0$ such that $S\left(  \lambda^{\ast},2\right)  =e\pi$,  then $S\left(  \lambda,2\right)  $ is not attained for any
$\lambda>\lambda^{\ast}$.
\end{lemma}

\begin{proof}
Assume that $\{u_{k}\}\subset H_{rad}^{1}\left(  \mathbb{R}^{2}\right)  $ is a
maximizing sequence of $S\left(  \lambda,2\right)  $, that is%
\[
\left\vert \left\vert u_{k}\right\vert \right\vert _{H^{1}\left(
\mathbb{R}^{2}\right)  }\leq1\text{, }\underset{k\rightarrow\infty}{\lim}%
\int_{\mathbb{R}^{2}}\left(  e^{4\pi u_{k}^{2}}-1-\lambda |u_{k}|^{2}\right)
dx=S\left(  \lambda,2\right)  \text{.}%
\]
Then, up to a subsequence, there exists $u_{\infty}\in H^{1}\left(
\mathbb{R}^{2}\right)  $ such that $u_{k}\rightharpoonup u_{\infty}$ and
$\left\vert \left\vert u_{\infty}\right\vert \right\vert _{H^{1}\left(
\mathbb{R}^{2}\right)  }\leq1$. Since $S\left(  \lambda,2\right)  >e\pi\geq4\pi-\lambda$,
Lemma \ref{CN} implies that $\{u_{k}\}$ is neither a normalized concentrating nor a normalized vanishing sequence. Then we claim that%
\begin{equation}
\underset{k\rightarrow\infty}{\lim}\int_{\mathbb{R}^{2}}\left(  e^{4\pi
u_{k}^{2}}-1-\lambda |u_{k}|^{2}\right)  dx=\int_{\mathbb{R}^{2}}\left(  e^{4\pi
u_{\infty}^{2}}-1-\lambda |u_{\infty}|^{2}\right)  dx\text{,}\label{3.2}
\end{equation}
The proof of the above equality follows a similar argument to that in \cite{liruf}, we briefly outline the proof framework here. By virtue of the non‑concentration of 
${u_k}$, we similarly obtain
\begin{equation}
\begin{aligned}
\lim_{k\to+\infty}\int_{\mathbb{R}^2}\big(e^{4\pi u_k^2}-1-\lambda|u_k|^2\big)dx
&=\int_{\mathbb{R}^2}\big(e^{4\pi u_\infty^2}-1-\lambda|u_\infty|^2\big)dx \\
&\quad +(4\pi-\lambda)\lim_{k\to+\infty}\int_{\mathbb{R}^2}\big(|u_k|^2-|u_\infty|^2\big)dx.
\end{aligned}
\label{3.4}
\end{equation}
Hence, it suffices to prove $u_k\to u_\infty$ in $L^2(\mathbb{R}^2)$. We assume that $u_\infty\neq0$, otherwise, we can deduce from (\ref{3.4}) that, $$S(\lambda,2)\leq4\pi-\lambda,$$ which leads to a contradiction. Set \[
\tau^2=\lim_{k\to+\infty}\frac{\displaystyle\int_{\mathbb{R}^2}u_k^2 dx}{\displaystyle\int_{\mathbb{R}^2} u_{\infty}^2 dx}.
\] Using analogous argument as in \cite{liruf}, we obtain that $$S(\lambda, 2) \geq S(\lambda, 2) + (\tau^2 - 1) \int_{\mathbb{R}^2} \left( e^{4\pi u_\infty^2} - 1 - 4\pi u_\infty^2 \right) dx.$$
Since $\int_{\mathbb{R}^2} \left( e^{4\pi u_\infty^2} - 1 - 4\pi u_\infty^2 \right) dx>0$, we have $\tau=1$. Then we obtain (\ref{3.2}), which implies $u_{\infty}$ is a maximizer.
\vskip0.1cm

Let $\lambda>\lambda^{\ast}$, we assume by contradiction that $S\left(
\lambda,2\right)  $  is attained by $u_{\lambda}\in H_{rad}^{1}\left(
\mathbb{R}^{2}\right)  $, then%
\begin{align*}
e\pi &  \leq S\left(  \lambda,2\right)  =\int_{\mathbb{R}^{2}}\left(  e^{4\pi
u_{\lambda}^{2}}-1-\lambda |u_{\lambda}|^{2}\right)  dx\\
&  <\int_{\mathbb{R}^{2}}\left(  e^{4\pi u_{\lambda}^{2}}-\lambda^{\ast
}|u_{\lambda}|^{2}-1\right)  dx\\
&  \leq S\left(  \lambda^{\ast},2\right)  =e\pi\text{,}%
\end{align*}
which is a contradiction. 
\end{proof}
By a similar argument, we  derive the following conclusion.
\begin{lemma}
\label{plarge}(1). For $p>2$, if $S\left(  \lambda,p\right)  >4\pi$, then $S\left(  \lambda,p\right)
$ is attained; 

(2). If there exists some
$\lambda^{\ast}>0$ such that $S\left(  \lambda^{\ast},p\right)  =4\pi$,  then $S\left(  \lambda,p\right)  $ is not attained for any
$\lambda>\lambda^{\ast}$.
\end{lemma}

\section{The attainability of $S(\lambda,p)$ and  the effect of sharp $L^p$ perturbation on the existence of extremals for $p>2$}
In this section, we study attainability of $S(\lambda,p)$ and the effect of $L^p$ perturbation on the existence of extremals in the case $p>2$, and give the proof of Theorem \ref{th1.2}. 

We first present the following attainability result.

\begin{lemma}\label{attainess}For $p>4$,  $S\left(  \lambda,p\right)  $ is always attained;  for $p\in(2,4]$, there exists some $\epsilon>0$ such that  
  $S\left(  \lambda,p\right)  $ is attained when $\lambda<\epsilon;$ while for $p=2$, there exists some $\epsilon>0$ such that  
  $S\left(  \lambda,p\right)  $ is attained when $0\leq\lambda<4\pi+\epsilon$.
\end{lemma}
\begin{proof}
    We first consider the case $p>2$.  From Lemma \ref{CN}, we have $$S(\lambda,p)\geq d_{nvl}=4\pi>e\pi=d_{ncl},$$  then the concentration phenomenon will not occur  in this case. Consequently, we may restrict our attention to the vanishing phenomenon.

When $p>4$, we fix a function $v\in H^{1}\left(  \mathbb{R}^{2}\right)  $, which is bounded,
spherically symmetric and non-increasing such that $\int_{\mathbb{R}^2}|v|^2dx=1$.   For $t>0$, we introduce a family of
functions given by
\[
\omega_{t}=\frac{t^{\frac{1}{2}}v\left(  t^{\frac{1}{2}}x\right)  }{\left(
1+t\left\vert \left\vert \nabla v\right\vert \right\vert _{2}^{2}\right)
^{\frac{1}{2}}}\text{.}%
\]
It is clear that $\left\vert \left\vert \omega_{t}\right\vert \right\vert
_{2}^{2}+\left\vert \left\vert \nabla\omega_{t}\right\vert \right\vert
_{2}^{2}=1$ and $\omega_{t}\rightarrow0$ uniformly in $\mathbb{R}^{2}$,  as $t\rightarrow0$.  Hence, we have%
\begin{align*}
S\left(  \lambda,p\right)   &  \geq\int_{\mathbb{R}^{2}}\left(  e^{4\pi
\omega_{t}^{2}}-1-\lambda\left\vert \omega_{t}\right\vert ^{p}\right)  dx\\
&  =\int_{\mathbb{R}^{2}}\left(  4\pi|\omega_{t}|^{2}+8\pi^2|\omega_{t}|^{4}-\lambda\left\vert \omega_{t}\right\vert ^{p}+o\left(
|\omega_{t}|^{4}\right)  \right)  dx\\
&  =\frac{4\pi}{1+t\left\vert \left\vert \nabla v\right\vert \right\vert
_{2}^{2}}+\frac{8\pi^2t\left\vert \left\vert
v\right\vert \right\vert _{4}^{4}}{\left(  1+t\left\vert \left\vert \nabla
v\right\vert \right\vert _{2}^{2}\right)  ^{2}}-\frac{\lambda t^{\frac{p}%
{2}-1}\left\vert \left\vert v\right\vert \right\vert _{p}^{p}}{\left(
1+t\left\vert \left\vert \nabla v\right\vert \right\vert _{2}^{2}\right)
^{\frac{p}{2}}}+o\left(  t\right)  \\
&  =4\pi+4\pi t\left\vert \left\vert \nabla v\right\vert \right\vert _{2}%
^{2}\left(  -1+2\pi\frac{\left\vert \left\vert v\right\vert \right\vert
_{4}^{4}}{\left\vert \left\vert \nabla v\right\vert \right\vert _{2}^{2}%
}\right)  -\lambda t^{\frac{p}{2}-1}\left\vert \left\vert v\right\vert
\right\vert _{p}^{p}+o\left(  t\right)  \text{,}%
\end{align*}
as $t\rightarrow0$. Recall that the following Gagliardo--Nirenberg--Sobolev best constant 
\[
B_{2}=\underset{u\in H^{1}\left(  \mathbb{R}^{2}\right)  ,u\neq0}{\sup}%
\frac{\left\vert \left\vert u\right\vert \right\vert _{4}^{4}}{\left\vert
\left\vert u\right\vert \right\vert _{2}^{2}\left\vert \left\vert \nabla
u\right\vert \right\vert _{2}^{2}}>\frac{1}{2\pi}%
\]
is attained by some function $U\in H^{1}\left(  \mathbb{R}^{2}\right)$  satisfying $\left\vert \left\vert
U\right\vert \right\vert _{2}^{2}=1$ (see \cite{Weinstein}).    Taking $v=U$, then we have%
\begin{equation}
S\left(  \lambda,p\right)  \geq4\pi+4\pi t\left\vert \left\vert \nabla
U\right\vert \right\vert _{2}^{2}\left(  -1+2\pi B_{2}\right)  -\lambda
t^{\frac{p}{2}-1}\left\vert \left\vert U\right\vert \right\vert _{p}%
^{p}+o\left(  t\right)  \text{,}\label{5.8}%
\end{equation}
as $t\rightarrow 0$. Letting $t$ small enough, we can obtain
\[
S\left(  \lambda,p\right)  >4\pi\text{.}%
\] 
Hence, $S\left(  \lambda,p\right)$ is aways attained for any $p>4$.

Now we consider the case  $2<p\leq4$. It follows from (\ref{5.8})
that $S\left(  \lambda,p\right)  $ can be attained provided
\[
4\pi t\left\vert \left\vert \nabla U\right\vert \right\vert _{2}^{2}\left(
-1+2\pi B_{2}\right)  -\lambda t^{\frac{p}{2}-1}\left\vert \left\vert
U\right\vert \right\vert _{p}^{p}+o\left(  t\right)  >0\text{,}%
\]
 which is equivalent to
\[
\lambda<\frac{4\pi t^{2-\frac{p}{2}}\left\vert \left\vert \nabla U\right\vert
\right\vert _{2}^{2}\left(  -1+2\pi B_{2}\right)  +o\left(  t^{2-\frac{p}{2}%
}\right)  }{\left\vert \left\vert U\right\vert \right\vert _{p}^{p}}\text{.}%
\]
Since $t>0$ and $B_{2}>\frac{1}{2\pi}$, we obtain
\[
\begin{aligned}
\frac{4\pi t^{2-\frac{p}{2}}\left\vert \left\vert \nabla U\right\vert
\right\vert _{2}^{2}\left(  -1+2\pi B_{2}\right)  +o\left(  t^{2-\frac{p}{2}%
}\right)  }{\left\vert \left\vert U\right\vert \right\vert _{p}^{p}}&\geq
\frac{1}{C_{p}}\left(  4\pi t^{2-\frac{p}{2}}\left\vert \left\vert \nabla
U\right\vert \right\vert _{2}^{4-p}\left(  -1+2\pi B_{2}\right)  +o\left(
t^{2-\frac{p}{2}}\right)  \right)\\& >0\text{,}%
\end{aligned}
\]
where in the first inequality we have used the Gagliardo–Nirenberg inequality $$\left\vert \left\vert U\right\vert \right\vert _{p}^{p}\leq
C_{p}\left\vert \left\vert U\right\vert \right\vert _{2}^{2}\left\vert
\left\vert \nabla U\right\vert \right\vert _{2}^{p-2}=C_{p}\left\vert
\left\vert \nabla U\right\vert \right\vert _{2}^{p-2}$$  for some $C_{p}>0$. Hence, there exists some $\epsilon>0$ such that $S(\lambda,p)$ is attained when $\lambda<\epsilon$. 

Finally, we demonstrate that  $S\left(  \lambda,2\right)  $ is attained when $\lambda$  is slightly larger than $4\pi$ by the test function argument.
Similar to that in \cite{liruf}, we set
\begin{equation}
u_{\varepsilon}\left(  x\right)  =\left\{
\begin{array}
[c]{cc}%
C-\frac{\log\left(  1+\pi\left\vert \frac{x}{\varepsilon}\right\vert
^{2}\right)  +\Lambda}{4\pi C}\text{,} & \left\vert x\right\vert
\leq L\varepsilon\text{,}\\
\frac{G\left(  \left\vert x\right\vert \right)  }{C}\text{,} & \left\vert
x\right\vert >L\varepsilon\text{,}%
\end{array}
\right.  \label{5.9}%
\end{equation}
where $L=(-\log \varepsilon)^2$, $G\left(  \left\vert x\right\vert \right)  $ is
the Green function of $  -\Delta+1  $ with singularity at $0$,
$\Lambda$ and $C$ are constants depending only on $\varepsilon$ satisfying \\
(i)  $C\rightarrow \infty \text{ and } L\varepsilon\rightarrow 0 \text{ as } \varepsilon \rightarrow 0$; \\
(ii)  $C-\frac{\log\left(  1+\pi L^{2}\right)  +\Lambda}{4\pi C}=\frac{G(L\varepsilon)}{C}$;\\
(iii)    $\frac{\log L}{C^2}\rightarrow 0, \text{ as } \varepsilon \rightarrow 0.$ 

To ensure that  $\left\vert \left\vert u_{\varepsilon}\right\vert
\right\vert _{H^{1}}=1$,  by direct computation, one has
 \begin{equation}\label{abou C}
     4\pi C^2=-1+4\pi A+ \log \pi-\log \varepsilon^2 +\phi
 \end{equation}
 with $\phi=O\left((L\varepsilon)^2 C^2 \log L+(L\varepsilon)^2\log^2 L\varepsilon+L^{-2} \right)$, where $A$  stands for the value of the regular part of the Green function $G$ at $0$. 
By (ii) and \eqref{abou C}, one can obtain $$\Lambda=-1+\phi.$$
Analogous to the analysis in \cite{liruf}, using  \eqref{abou C},  we have
\begin{equation}
\int_{B_{L\varepsilon}}\left(  e^{4\pi u_{\varepsilon}^{2}}-1\right)
dx\geq\pi e^{4\pi A+1}+O\left(  L^{-2}\right)  \text{,}\label{85}%
\end{equation}
and
\begin{equation}
\int_{\mathbb{R}^{2}\backslash B_{L\varepsilon}}\left(  e^{4\pi u_{\varepsilon
}^{2}}-1\right)  dx\geq\frac{4\pi\int
_{\mathbb{R}^{2}\backslash B_{L\varepsilon}}G^{2}dx}{C^{2}}+
\frac{8\pi^{2}\int
_{\mathbb{R}^{2}\backslash B_{L\varepsilon}}G^{4}dx}{C^{4}}
 dx\text{.}%
\label{87}%
\end{equation}
On the other hand, a direct computation yields
\begin{equation}
\int_{\mathbb{R}^{2}} |u_{\varepsilon}|^{2}dx=\left(\int_{B_{L\varepsilon}}+\int_{\mathbb{R}^{2}\backslash B_{L\varepsilon}}\right)|u_{\varepsilon}|^{2}dx=O\left(  C^{2}\left(
L\varepsilon\right)  ^{2}\log L\right)  + \frac{\int_{\mathbb{R}^{2}\backslash B_{L\varepsilon}}G^{2}dx}{C^{2}%
} \text{.}\label{86}%
\end{equation}
Combining (\ref{abou C}), (\ref{85}), (\ref{87}) and (\ref{86}), we get%
\begin{align}
\int_{\mathbb{R}^{2}}\left(  e^{4\pi u_{\varepsilon}^{2}}-1-\lambda
|u_{\varepsilon}|^{p}\right)  dx  & \geq\pi e^{4\pi A+1}+O\left(  L^{-2}\right)
\label{88}\\
& +\frac{(4\pi-\lambda)\int_{\mathbb{R}^{2}\backslash B_{L\varepsilon}}G^{2}%
dx}{C^{2}}+\frac{8\pi^{2}\int
_{\mathbb{R}^{2}\backslash B_{L\varepsilon}}G^{4}dx}{C^{4}} 
\text{.}\nonumber
\end{align}
From  Theorem 1.5 in \cite{Nguyen-JFA}, we have $\pi e^{4\pi
A+1}=\pi e$.
Hence, $S\left(  \lambda,2\right)  $ is attained provided
\begin{equation}\label{condition for}
    O\left(  L^{-2}\right)
+\frac{(4\pi-\lambda)\int_{\mathbb{R}^{2}\backslash B_{L\varepsilon}}G^{2}%
dx}{C^{2}}+\frac{8\pi^{2}\int
_{\mathbb{R}^{2}\backslash B_{L\varepsilon}}G^{4}dx}{C^{4}} 
dx>0\text{.}%
\end{equation}
Specifically, when $\epsilon$ small enough, if
\[
\lambda < 4\pi + \frac{
    \frac{1}{C^2} \int_{\mathbb{R}^2 \setminus B_{L\epsilon}} \frac{(4\pi)^2 G^4}{2} dx
    + O\left(C^2 L^{-2}\right)
}{
    \int_{\mathbb{R}^2 \setminus B_{L\epsilon}} G^2 dx
}>4\pi,
\]
where the last inequality follows from $C^{4}\sim(-\log \varepsilon)^2=L$ by (\ref{abou C}), then $S\left(  \lambda,2\right)>\pi e  $, and Lemma \ref{57} yields the attainability of $S\left(  \lambda,2\right)  $, and the proof is finished.
\end{proof}

Now, we complete the
\begin{proof}[Proof of Theorem \ref{th1.2}]
From Lemma \ref{attainess}, we only need to show the existence of threshold $\lambda^*$  for existence and nonexistence when $2<p\leq4$. We prove this  by contradiction: assume that  $S\left(  \lambda_{k},p\right)  $  can be attained for some sequence
$\lambda_{k}\rightarrow\infty$ as $k\rightarrow\infty$, and denote by $u_{k}$ the
corresponding maximizer of $S\left(  \lambda_{k},p\right)
$. 

From Lemma \ref{lem2.4}, we have%
\begin{equation}
S\left(  0,p\right)  -\lambda_{k}\int_{\mathbb{R}^{2}}\left\vert
u_{k}\right\vert ^{p}dx\geq\int_{\mathbb{R}^{2}}\left(  e^{4\pi u_{k}^{2}%
}-1-\lambda_{k}\left\vert u_{k}\right\vert ^{p}\right)  dx=S\left(
\lambda_{k},p\right)  \geq 4\pi\text{,}\label{3.1}
\end{equation}
which implies that $\lambda_{k}\int_{\mathbb{R}^{2}}\left\vert u_{k}%
\right\vert ^{p}dx<C$ for some constant $C>0$. Hence we can deduce $u_{k}\rightarrow0$ in
$L^{p}\left(  \mathbb{R}^{2}\right)$. 

Since $S\left(  \lambda,p\right)
>d_{ncl}=e\pi$ for $p>2$, the sequence $\{u_{k}\}_{k}$ cannot be an NCS 
in $H_{rad}^{1}\left(  \mathbb{R}^{2}\right) $. Therefore, it is an  NVS in
$H_{rad}^{1}\left(  \mathbb{R}^{2}\right)  $, and $\int_{\mathbb{R}^{2}}\left\vert
u_{k}\right\vert ^{2}dx\rightarrow1$ as $k\rightarrow\infty$. Obviously,
\begin{align}
S\left(  \lambda_{k},p\right)   &  =\int_{\mathbb{R}^{2}}\left(  e^{4\pi
u_{k}^{2}}-1-\lambda_{k}\left\vert u_{k}\right\vert ^{p}\right)  dx\nonumber\\
&  =\left(  \int_{\mathbb{R}^{2}}4\pi |u_{k}|^{2}dx\right)  -\left(
\int_{\mathbb{R}^{2}}\lambda_{k}\left\vert u_{k}\right\vert ^{p}dx\right)
+\left(  \sum\limits_{j=2}^{\infty}\frac{\left(  4\pi\right)  ^{j}}{j!}%
\int_{\mathbb{R}^{2}}|u_{k}|^{2j}dx\right)  \label{5.7}\\
&  =:I_{1}-I_{2}+I_{3}\text{.}\nonumber
\end{align}
Since $\{u_{k}\}_{k}$ is an NVS, similar to \eqref{5.2}, we can obtain
\begin{equation}
I_{3}\leq C\int_{\mathbb{R}^{2}}\left\vert u_{k}\right\vert ^{p}dx\label{5.6}%
\end{equation}
for some $C>0$. Substituting (\ref{5.6})  into (\ref{5.7}), we have%
\[
S\left(  \lambda_{k},p\right)  \leq4\pi\int_{\mathbb{R}^{2}}|u_{k}|%
^{2}dx-\lambda_{k}\int_{\mathbb{R}^{2}}\left\vert u_{k}\right\vert
^{p}dx+C\int_{\mathbb{R}^{2}}\left\vert u_{k}\right\vert ^{p}dx<4\pi\text{,}%
\]
for sufficiently large $\lambda_{k}$, this yields a contradiction, since $S\left(  \lambda
_{k},p\right)  \geq4\pi$ whenever $2<p\leq4$. Combining this with Lemma \ref{plarge} and Theorem 1.7 in \cite{Chenluzhu}, we complete the proof.
\end{proof}

\section{The effect of sharp $L^p$ perturbation on the existence of extremals for  $p=2$}

Recall that in the case  $p=2$,   we have shown in Lemma \ref{attainess} that there exists a small $\varepsilon_0 > 0$ such that the perturbed Trudinger-Moser inequalities are attained for $\lambda \in [0, 4\pi + \varepsilon_0) $. In this section, we will establish the
existence of a threshold ${\lambda}^{\ast} \in (4\pi, +\infty)$ that determines the existence and non-existence of extremals. We argue by contradiction. Suppose, to the contrary, that no such threshold exists for $S(\lambda,2)$. By investigating the asymptotic expansion of the standard Sobolev norm for the sequence of extremal functions $\{u_{\lambda}\}$, we will obtain a contradiction with the fact that $\left\vert\left\vert u_{\lambda}\right\vert\right\vert_{H^{1}\left(
\mathbb{R}^{2}\right)  }=1$  as $\lambda
\rightarrow\infty$, which was proved in Proposition \ref{114}.
\vskip0.1cm

\vskip0.1cm

Let $\{\lambda_k\}$ satisfy $\lambda_k\to\infty$, and let $u_k=u_{\lambda_k}$
be the corresponding extremals of $S(\lambda_k,2)$ normalized by
\[
\|u_k\|_{H^1(\mathbb R^2)}=1,
\]
with
\[
u_k\in H^1_{\rm rad}(\mathbb R^2),\qquad
u_k\ge0,\qquad
\frac{\partial u_k}{\partial r}<0.
\]

From the positivity of the maximized functional,
\[
S(\lambda_k,2)
=
\int_{\mathbb R^2}
\left(
e^{4\pi u_k^2}-1-\lambda_k u_k^2
\right)\,dx
>0,
\]
we obtain
\[
\lambda_k
\int_{\mathbb R^2}u_k^2\,dx
\le
\int_{\mathbb R^2}
\left(e^{4\pi u_k^2}-1\right)\,dx.
\]
Since $\|u_k\|_{H^1}=1$, the Trudinger--Moser inequality yields
\[
\int_{\mathbb R^2}
\left(e^{4\pi u_k^2}-1\right)\,dx
\le C,
\]
where $C$ is independent of $k$. Consequently,
\begin{equation}\label{add76}
    \int_{\mathbb R^2}|u_k|^2\,dx
=
O(\lambda_k^{-1})
\rightarrow0, \text{as } k\rightarrow \infty.
\end{equation}

Since $\{u_k\}$ is bounded in $H^1(\mathbb R^2)$, there exists
$u_\infty\in H^1(\mathbb R^2)$ such that, after passing to a subsequence,
\[
u_k\rightharpoonup u_\infty
\qquad\text{weakly in }H^1(\mathbb R^2).
\]
On the other hand,
\[
\|u_k\|_{L^2(\mathbb R^2)}\rightarrow0,
\]
so necessarily $u_\infty\equiv0$. Hence
\[
u_k\rightharpoonup0
\qquad\text{in }H^1(\mathbb R^2).
\]

Next we show that $u_k(x)\to0$ for every $x\neq0$.
Fix $r>0$. Since each $u_k$ is radial and strictly decreasing, thst is,
\[
u_k(s)\ge u_k(r),
\text{ for }
0\le s\le r.
\]
Therefore,
\[
\int_{B_r\setminus B_{r/2}}u^2_k\,dx
\ge
|B_r\setminus B_{r/2}|\,u^2_k(r).
\]
It follows that
\[
u^2_k(r)
\le
\frac{1}{|B_r\setminus B_{r/2}|}
\int_{\mathbb R^2}u_k^2\,dx
\rightarrow0.
\]
Hence
\[
u_k(x)\rightarrow0,
\qquad
x\in\mathbb R^2\setminus\{0\}.
\]

We now prove that the maximum of $u_k$ must blow up.
Suppose instead that
\[
u_k(0)\le M
\]
for some constant $M>0$.
Since $u_k$ is radial and decreasing,
\[
0\le u_k(x)\le M
\qquad
\text{for every }x\in\mathbb R^2.
\]
Consequently,
\[
e^{4\pi u_k^2}-1
\le
\frac{e^{4\pi M^2}-1}{M^2}\,u_k^2.
\]
Therefore,
\[
\int_{\mathbb R^2}
\left(e^{4\pi u_k^2}-1\right)\,dx
\le
C(M)
\int_{\mathbb R^2}u_k^2\,dx
\rightarrow0,
\]
where the constant $C(M)$ is interpreted as $4\pi$ if $M=0$.

However, Lemma~\ref{lem2.4} gives
\begin{equation}
   e\pi
\le
\liminf_{k\to\infty}
\int_{\mathbb R^2}
\left(e^{4\pi u_k^2}-1-\lambda_k u_k^2\right)\,dx
\le
\liminf_{k\to\infty}
\int_{\mathbb R^2}
\left(e^{4\pi u_k^2}-1\right)\,dx, \label{add763}
\end{equation}which is impossible.
Hence the sequence $\{u_k(0)\}$ cannot remain bounded, and therefore
\[
u_k(0)
=
\max_{x\in\mathbb R^2}u_k(x)
\rightarrow\infty.
\]

Thus,
\[
u_k\rightharpoonup0
\quad\text{in }H^1(\mathbb R^2),
\qquad
u_k(x)\to0
\quad
(x\neq0),
\]
while the maximum value satisfies
\[
\max_{\mathbb R^2}u_k=u_k(0)\rightarrow\infty.
\]

By the Lagrange multiplier theorem, $u_{k}$ satisfies%

\begin{equation}
-\Delta u_{k}+u_{k}=\frac{4\pi}{E_{k}}\left(  u_{k}e^{4\pi u_{k}^{2}}%
-\frac{\lambda_{k}}{4\pi}u_{k}\right)  \text{, }\label{15}%
\end{equation}
where
\[
E_{k}:=4\pi\int_{\mathbb{R}^{2}}\left(  u_{k}^{2}e^{4\pi u_{k}^{2}}%
-\frac{\lambda_{k}}{4\pi}|u_{k}|^{2}\right)  dx\text{.}%
\]
Set $v_{k}:=\sqrt{4\pi}u_{k}$, then $v_{k}$ satisfies%
\begin{equation}
-\Delta v_{k}+v_{k}=\frac{4\pi}{E_{k}}\left(  v_{k}e^{v_{k}^{2}}-\frac
{\lambda_{k}}{4\pi}v_{k}\right)  \text{,}\label{14}%
\end{equation}
with%
\begin{equation}
\left\vert \left\vert v_{k}\right\vert \right\vert _{H^{1}\left(
\mathbb{R}^{2}\right)  }^{2}=4\pi\text{, }E_{k}=\int_{\mathbb{R}^{2}}\left(
v_{k}^{2}e^{v_{k}^{2}}-\frac{\lambda_{k}}{4\pi}|v_{k}|^{2}\right)
dx\text{.}\label{13}%
\end{equation}
Thus, we can rewrite $\left\vert \left\vert v_{k}\right\vert \right\vert
_{H^{1}\left(  \mathbb{R}^{2}\right)  }$ as
\begin{equation}
4\pi=\left\vert \left\vert v_{k}\right\vert \right\vert _{H^{1}\left(
\mathbb{R}^{2}\right)  }^{2}=\frac{4\pi}{E_{k}}\int_{\mathbb{R}^{2}}\left(
v_{k}^{2}e^{v_{k}^{2}}-\frac{\lambda_{k}}{4\pi}|v_{k}|^{2}\right)
dx:=I_{E}-I_{P}.\label{118}%
\end{equation}

In the following,  we investigate the asymptotic behavior of $v_{k}$ as $k\rightarrow\infty$, and provide sharp estimates for the terms $I_{E}$ and $I_{P}$ on the right-hand side of (\ref{118}), respectively. 

\begin{proposition}
\label{223}There exists some constant $C>0$ such that%
\[
\frac{\lambda_{k}}{E_{k}}\int_{\mathbb{R}^{2}}|v_{k}|^{2}dx\geq\frac
{C\lambda_{k}}{\gamma_{k}^{4}}\text{.}%
\]
\end{proposition}

\begin{proposition}
\label{129}It holds that
\begin{equation}
\frac{4\pi}{E_{k}}\int_{\mathbb{R}^{2}}v_{k}^{2}e^{v_{k}^{2}}dx=4\pi+O\left(
\frac{1}{\gamma_{k}^{4}}\right)  \text{.}\label{55}%
\end{equation}

\end{proposition}

Since the proofs of Proposition~\ref{223} and Proposition~\ref{129} rely on the asymptotic behavior of the sequence $u_k$, and are rather lengthy and technical, we postpone them to Sections~\ref{poly} and~\ref{esti Exponential}, respectively.
  
   Based on Proposition~\ref{223} and Proposition~\ref{129}, we can obtain the following expansion formula
for $\left\vert \left\vert u_{k}\right\vert \right\vert _{H^{1}\left(
\mathbb{R}^{2}\right)  }^{2}$. \

\begin{proposition}
\label{56}There exist positive constants $C_{1}$, $C_{2}$ such that as
$\lambda_{k}\rightarrow\infty$,
\[
\left\vert \left\vert u_{k}\right\vert \right\vert _{H^{1}\left(
\mathbb{R}^{2}\right)  }^{2}\leq1-\lambda_{k}\frac{C_{1}}{\gamma_{k}^{4}%
}+\frac{C_{2}}{\gamma_{k}^{4}}+o\left(  \gamma_{k}^{-4}\right)  \text{.}%
\]

\end{proposition}

\begin{proof}
Using (\ref{118}), we have from Propositions \ref{223} and \ref{129} that there exist positive constants $C_1$ and  $C_2$ such that
\begin{align*}
\left\vert \left\vert u_{k}\right\vert \right\vert _{H^{1}\left(
\mathbb{R}^{2}\right)  }^{2} &  =\frac{1}{4\pi}\left\vert \left\vert
v_{k}\right\vert \right\vert _{H^{1}\left(  \mathbb{R}^{2}\right)  }^{2}\\
&  =\frac{1}{4\pi}\left[  \frac{4\pi}{E_{k}}\int_{\mathbb{R}^{2}}\left(
v_{k}^{2}e^{v_{k}^{2}}-\frac{\lambda_{k}}{4\pi}|v_{k}|^{2}\right)  dx\right]  \\
&  \leq\frac{1}{4\pi}\left[  4\pi+\frac{C_{2}}{\gamma_{k}^{4}}+o\left(
\gamma_{k}^{-4}\right)  -\frac{C_{1}\lambda_{k}}{\gamma_{k}^{4}}\right]  \\
&  \leq1-\lambda_{k}\frac{C_{1}}{\gamma_{k}^{4}}+\frac{C_{2}}{\gamma_{k}^{4}%
}+o\left(  \gamma_{k}^{-4}\right)  \text{,}%
\end{align*}
this completes the proof.\
\end{proof}

Using this expansion formula, we can give the 
\begin{proof}[Proof of Theorem \ref{th1.1}]

According to Proposition \ref{114}, it follows that all maximizers
$u_{\lambda}$ of $S\left(  \lambda,2\right)  $ satisfy $\left\vert \left\vert
u_{\lambda}\right\vert \right\vert _{H^{1}\left(  \mathbb{R}^{2}\right)  }%
^{2}=1$. However, Proposition \ref{56} asserts that, if there exists a
maximizer $u_{\lambda}$ for sufficiently large $\lambda$, then $\left\vert
\left\vert u_{\lambda}\right\vert \right\vert _{H^{1}\left(  \mathbb{R}%
^{2}\right)  }^{2}<1$, which is a contradiction and there is no maximizer of
$S\left(  \lambda,2\right)  $ for sufficiently large $\lambda$. Combining Lemma
\ref{lem2.4} and Lemma \ref{57}, we see that $S\left(  \lambda,2\right)  =e\pi$ for
sufficiently large $\lambda$. Define $\lambda^{\ast}$ as follows%
\[
\lambda^{\ast}:=\sup\{\lambda>4\pi+\varepsilon_{0}\mid S\left(  \lambda,2
\right)  \text{ could be achieved}\}\text{.}%
\]
We claim that $S\left(  \lambda^{\ast},2\right)  =e\pi$. If $S\left(
\lambda^{\ast},2\right)  >e\pi$, then $S\left(  \lambda^{\ast},2\right)  $ can be
achieved by some $u_{\lambda^{\ast}}$ through Lemma \ref{57}. For
$\lambda>\lambda^{\ast}$ and sufficiently close to $\lambda^{\ast}$, we have%
\[
S\left(  \lambda,2\right)  \geq\int_{\mathbb{R}^{2}}\left(  e^{4\pi
u_{\lambda^{\ast}}^{2}}-1-\lambda |u_{\lambda^{\ast}}|^{2}\right)
dx>e\pi\text{,}%
\]
which implies $S\left(  \lambda,2\right)  $ is achieved for $\lambda$ slightly
bigger than $\lambda^{\ast}$, this contradicts the definition of
$\lambda^{\ast}$. Hence the claim is proved. Then as a direct consequence of
Lemma \ref{57}, Theorem \ref{th1.1} holds with the definition of $\lambda^{\ast}$.
\end{proof}

\section{Asymptotic behavior of $u_k$ and Estimates for the Lebesgue integral $I_{P}$ \label{poly}}

In this section, we  study the limiting behavior of $\left\{
v_{k}\right\}  $ near and far away from the blow-up point $0$, and estimate
for the Lebesgue integral $I_{P}$ in (\ref{118}). To this end, we require several lemmas and propositions.

\begin{lemma}
\label{17}The $E_{k}$ in (\ref{13}) satifies
\[
\underset{k\rightarrow\infty}{\lim\inf}E_{k}>0\text{.}%
\]
Moreover, $E_k\to\infty$. 
\end{lemma}

\begin{proof}

From the elementary inequality $e^{t} - 1 \leq t e^{ t}$ for $t \geq 0$, we obtain
\begin{align*}
  \int_{\mathbb{R}^{2}}\left(  e^{v_{k}^{2}}-1-\frac{\lambda_k}{4\pi} |v_{k}|%
^{2}\right)  dx
 \leq\int_{\mathbb{R}^{2}}\left(  v_{k}^{2}e^{ v_{k}^{2}}-\frac{\lambda_k}{4\pi} 
|v_{k}|^{2}\right)  dx= E_{k}\text{.}%
\end{align*}
On the other hand, by Lemma \ref{lem2.4}, $$\int_{\mathbb{R}^{2}}\left(  e^{v_{k}^{2}}-1-\frac{\lambda_k}{4\pi} |v_{k}|%
^{2}\right)  dx=\int_{\mathbb{R}^{2}%
}\left(  e^{4\pi u_{k}^{2}}-1-\lambda_{k}|u_{k}|^{2}\right)  dx=S(\lambda_k,2)\geq e\pi.$$
Hence, we have\[
\underset{k\rightarrow\infty}{\lim\inf}E_{k}>0\text{.}%
\]

Now, we show $E_k\to\infty$, as $k\rightarrow \infty$. Let
\[
A_k:=\int_{\mathbb R^2} v_k^2 e^{v_k^2}\,dx,
\qquad
B_k:=\frac{\lambda_k}{4\pi}\int_{\mathbb R^2}v_k^2\,dx .
\]
Then
\[
E_k=A_k-B_k.
\]
Since
\[
\int_{\mathbb R^2}v_k^2\,dx
=
4\pi\int_{\mathbb R^2}u_k^2\,dx
=
O(\lambda_k^{-1}),
\]
we have
\[
B_k
=
\frac{\lambda_k}{4\pi}
\int_{\mathbb R^2}v_k^2\,dx
=
O(1).
\]
Therefore, to prove $E_k\to\infty$, it suffices to prove that
\[
A_k\to\infty.
\]

Suppose, by contradiction, that $A_k$ is bounded along a subsequence. Then
there exists a constant $C>0$ such that
\[
\int_{\mathbb R^2}v_k^2e^{v_k^2}\,dx\le C .
\]
We claim that this implies
\[
\int_{\mathbb R^2}(e^{v_k^2}-1)\,dx\to0.
\]

Indeed, fix $L>0$ and split
\[
\int_{\mathbb R^2}(e^{v_k^2}-1)\,dx
=
\int_{\{v_k^2\le L\}}(e^{v_k^2}-1)\,dx
+
\int_{\{v_k^2>L\}}(e^{v_k^2}-1)\,dx .
\]
On the set $\{v_k^2\le L\}$, we have
\[
e^{v_k^2}-1\le C_L v_k^2,
\]
where
\[
C_L:=\max_{0\le t\le L}\frac{e^t-1}{t}.
\]
Hence
\[
\int_{\{v_k^2\le L\}}(e^{v_k^2}-1)\,dx
\le
C_L\int_{\mathbb R^2}v_k^2\,dx
\to0.
\]
On the set $\{v_k^2>L\}$, since $t>L$ implies
\[
e^t-1\le e^t\le \frac{t}{L}e^t,
\]
we obtain
\[
e^{v_k^2}-1
\le
\frac1L v_k^2e^{v_k^2}.
\]
Therefore
\[
\int_{\{v_k^2>L\}}(e^{v_k^2}-1)\,dx
\le
\frac1L
\int_{\mathbb R^2}v_k^2e^{v_k^2}\,dx
\le
\frac{C}{L}.
\]
It follows that
\[
\limsup_{k\to\infty}
\int_{\mathbb R^2}(e^{v_k^2}-1)\,dx
\le
\frac{C}{L}.
\]
Letting $L\to\infty$, we conclude that
\[
\int_{\mathbb R^2}(e^{v_k^2}-1)\,dx\to0.
\]

Consequently,
\[
\begin{aligned}
S(\lambda_k,2)
&=
\int_{\mathbb R^2}
\left(
e^{v_k^2}-1-\frac{\lambda_k}{4\pi}v_k^2
\right)\,dx  \\
&\le
\int_{\mathbb R^2}(e^{v_k^2}-1)\,dx
\to0.
\end{aligned}
\]
This contradicts Lemma~\ref{lem2.4}, which gives
\[
e\pi
\le
\liminf_{k\to\infty}S(\lambda_k,2).
\]
Therefore $A_k$ cannot remain bounded, and hence
\[
A_k\to\infty.
\]
Since $B_k=O(1)$, we finally obtain
\[
E_k=A_k-B_k\to\infty,
\]the proof is finished.\end{proof}

In what follows, we sharpen the estimates \eqref{add762} and \eqref{add76} as $k\rightarrow \infty$ .

\begin{lemma} 
We have%
\begin{equation}
\underset{k\rightarrow\infty}{\lim}\int_{\mathbb{R}^{2}}\left(  e^{v_{k}^{2}%
}-1\right)  dx=e\pi\label{19}%
\end{equation}
and
\begin{equation}
\underset{k\rightarrow\infty}{\lim}\lambda_{k}\int_{\mathbb{R}^{2}}|v_{k}|%
^{2}dx=0\text{.} \label{20}%
\end{equation}

\end{lemma}

\begin{proof}
For any small $\varepsilon>0$, we choose a cut-off function such that%
\[
\phi_{\varepsilon}\in C^{\infty}\left(  \mathbb{R}^{2}\right)  \text{, }%
\phi_{\varepsilon}=0\text{ in }B_{\frac{\varepsilon}{2}}\text{, }%
\phi_{\varepsilon}\left(  x\right)  =1\text{ in }\mathbb{R}^{2}\backslash
B_{\varepsilon},%
\]
multiplying (\ref{15}) by $\phi_{\varepsilon}u_{k}$ and intergrating by parts,
we have
\begin{align*}
&\int_{\mathbb{R}^{2}\backslash B_{\frac{\varepsilon}{2}}}\left(  \nabla
u_{k}\nabla\left(  \phi_{\varepsilon}u_{k}\right)  +u_{k}\left(
\phi_{\varepsilon}u_{k}\right)  \right)  dx \\ &\leq\frac{4\pi}{E_{k}}%
\int_{\mathbb{R}^{2}\backslash B_{\frac{\varepsilon}{2}}}\left(  u_{k}e^{4\pi
u_{k}^{2}}-\frac{\lambda_{k}}{4\pi}u_{k}\right)  \phi_{\varepsilon}%
u_{k}dx\\
&  \leq\frac{4\pi}{E_{k}}\int_{\mathbb{R}^{2}\backslash B_{\frac{\varepsilon
}{2}}}\phi_{\varepsilon}u_{k}^{2}e^{4\pi u_{k}^{2}}dx\leq\frac{C}{E_{k}}%
\int_{\mathbb{R}^{2}}|u_{k}|^{2}dx.  
\end{align*}
It follows from (\ref{add76}) and Lemma \ref{17}  that%
\begin{equation}\label{concentrat}
    \underset{k\rightarrow\infty}{\lim}\int_{\mathbb{R}^{2}\backslash
B_{\varepsilon}}\left(  \left\vert \nabla u_{k}\right\vert ^{2}+|u_{k}|%
^{2}\right)  dx=0 \text{ for any } \varepsilon>0.
\end{equation}
Thus $\{u_{k}\}$ is a normalized concentrating sequence in $H^1(\mathbb{R}^{2})$. By \eqref{add763} and Ruf's result in \cite{Ruf}, we have
\begin{align*}
e\pi   \leq\underset{k\rightarrow\infty}{\lim\inf}\int_{\mathbb{R}^{2}%
}\left(  e^{4\pi u_{k}^{2}}-1\right)  dx \leq\underset{k\rightarrow\infty}{\lim\sup}\int_{\mathbb{R}^{2}}\left(
e^{4\pi u_{k}^{2}}-1\right)  dx\leq e\pi\text{.}%
\end{align*}
This together with Lemma \ref{lem2.4} yields
\begin{align*}
e\pi &  \leq\underset{k\rightarrow\infty}{\lim\inf}\int_{\mathbb{R}^{2}%
}\left(  e^{4\pi u_{k}^{2}}-1-\lambda_{k}|u_{k}|^{2}\right)  dx\\
&  \leq\underset{k\rightarrow\infty}{\lim\sup}\int_{\mathbb{R}^{2}}\left(
e^{4\pi u_{k}^{2}}-1-\lambda_{k}|u_{k}|^{2}\right)  dx\\
&  \leq e\pi\text{.}%
\end{align*}
Recalling that $v_{k}=\sqrt{4\pi}u_{k}$, we obtain (\ref{19}) and
(\ref{20}).
\end{proof}

We set  
 
\[
\gamma_{k}:=v_{k}\left(  0\right)  =\max v_{k}\left(  x\right)  \text{.}%
\]

\begin{equation}
r_{k}:=\frac{\sqrt{E_{k}}}{\sqrt{\pi}}\gamma_{k}^{-1}e^{-\frac{\gamma_{k}^{2}%
}{2}}\text{,} \label{3.5}%
\end{equation}
and define%
\[
\left\{
\begin{aligned}
\phi_k(y) &:= \gamma_k \left( v_k(r_k y) - \gamma_k \right), \\
\psi_k(y) &:= \gamma_k^{-1} v_k(r_k y),
\end{aligned}
\right.
\]
and $\Omega_{k}=\{y\in\mathbb{R}^{2}:r_{k}y\in B_{1}\}$. 
By direct computation, we see that $\phi_{k}$ and
$\psi_{k}$ satisfy%
\begin{equation}
-\Delta\phi_{k}=4\left(  \psi_{k}e^{\phi_{k}\left(  1+\psi_{k}\right)  }%
-\frac{\lambda_{k}}{4\pi}e^{-\gamma_{k}^{2}}\psi_{k}\right)  -\frac{E_{k}}%
{\pi}e^{-\gamma_{k}^{2}}\psi_{k}\text{ in }\Omega_{k}\text{} \label{21}%
\end{equation}%
and\begin{equation}
-\Delta\psi_{k}=\frac{4}{\gamma_{k}^{2}}\left(  \psi_{k}e^{\gamma_{k}%
^{2}\left(  \psi_{k}^{2}-1\right)  }-\frac{\lambda_{k}}{4\pi}e^{-\gamma
_{k}^{2}}\psi_{k}\right)  -\gamma_{k}^{-2}\frac{E_{k}}{\pi}e^{-\gamma_{k}^{2}%
}\psi_{k} \text{ in }\Omega_{k}\text{,} \label{22}%
\end{equation} respectively.
\vskip0.1cm

Since both equations above involve $E_k$,  $r_k$ and the  perturbation parameter $\lambda_k$(which tends to $+\infty$ as
$k\rightarrow\infty$ ), we need to estimate these quantities to investigate the asymptotic behavior of  $\phi_{k}$ and $\psi_{k}$ as $k\rightharpoonup \infty$.

\begin{proposition}\label{prop:Ek-growth}
We have
\[
E_k
=
o\!\left(e^{\gamma_k^2/2}\right).
\]
Consequently,
\[
r_k\rightarrow0.
\]
\end{proposition}

\begin{proof}
Since $v_k$ is nonnegative, radial, and monotonically decreasing, we have
\[
v^2_k(r)
\le
\frac{1}{|B_r|}
\int_{B_r}v_k^2\,dx
=
\frac{1}{\pi r^2}
\int_{B_r}v_k^2\,dx.
\]
Using the normalization
\[
\|v_k\|_{L^2(\mathbb R^2)}^2
\le
\|v_k\|_{H^1(\mathbb R^2)}^2
=
4\pi,
\]
it follows that
\[
v^2_k(r)
\le
\frac{1}{\pi r^2}
\int_{\mathbb R^2}v_k^2\,dx
\le
\frac{4}{r^2}.
\]
Consequently,
\[
v_k(x)\le \frac{2}{|x|}
\qquad\text{for every }x\neq0.
\]
In particular, choosing any fixed constant $L\ge2$, we obtain
 \begin{equation}\label{add774}
     v_k(x)\le1,
\qquad
x\in\mathbb R^2\setminus B_L,
 \end{equation}
uniformly for all $k$.
Note
\[
\int_{B_{L}}\left\vert \nabla\left(  v_{k}-v_{k}\left(  L\right)  \right)
^{+}\right\vert ^{2}dx\leq4\pi\text{.}%
\]
By the classical Trudinger-Moser inequality, we get%
\[
\int_{B_{L}}e^{\left[  \left(  v_{k}-v_{k}\left(  L\right)  \right)
^{+}\right]  ^{2}}\leq C\left(  L\right)  \text{.}%
\]
Clearly, for any $p<1$, we can find a constant $C\left(  p\right)  $ such
that
\[
pv_{k}^{2}\leq\left[  \left(  v_{k}-v_{k}\left(  L\right)  \right)
^{+}\right]  ^{2}+C\left(  p\right)  \text{,}%
\]
and then we get%
\begin{equation}\label{bound for finit ball}
    \int_{B_{L}}e^{pv_{k}^{2}}dx\leq C\text{.}%
\end{equation}
Hence, 
\begin{align*}
E_{k}e^{-\frac{1}{2}\gamma_{k}^{2}}  &  =e^{-\frac{1}{2}\gamma_{k}^{2}}%
\int_{\mathbb{R}^{2}}\left(  v_{k}^{2}e^{v_{k}^{2}}-\frac{\lambda_{k}}{4\pi
}|v_{k}|^{2}\right)  dx\\
&  \leq C\left[  \int_{B_{L}}v_{k}^{2}e^{\frac{1}{2}v_{k}^{2}}dx+e^{-\frac
{1}{2}\gamma_{k}^{2}}\int_{\mathbb{R}^{2}\backslash B_{L}}|v_{k}|^{2}dx\right] 
\text{.}%
\end{align*}
Since $v_{k}\rightarrow0$ in $L^{q}\left(  B_{L}\right)  $ for any $q>1$, using H\"{o}lder's inequality and \eqref{bound for finit ball},  we have
$E_{k}=o\left(  e^{\frac{1}{2}\gamma_{k}^{2}}\right)  $ and then
$r_{k}\rightarrow0$ as $k\rightarrow\infty$.     
\end{proof}

\begin{lemma}\label{lem:lambda-bound}
There holds
\begin{equation}
\underset{k\rightarrow\infty}{\lim}\frac{\lambda_{k}}{\pi}e^{-\gamma_{k}^{2}%
}=0\text{.} \label{23}%
\end{equation}
\end{lemma}

\begin{proof}
We first establish a uniform bound for the quantity
$\lambda_k e^{-\gamma_k^2}$. Recall that $\psi_k$ satisfies
\[
-\Delta\psi_k
=
\frac4{\gamma_k^2}
\left(
\psi_k
e^{\gamma_k^2(\psi_k^2-1)}
-
\frac{\lambda_k}{4\pi}
e^{-\gamma_k^2}\psi_k
\right)
-
\frac{E_k}{\pi\gamma_k^2}
e^{-\gamma_k^2}\psi_k
\]
in $\Omega_k$. Since
\[
\psi_k(y)
=
\frac{v_k(r_ky)}{\gamma_k},
\qquad
\gamma_k=v_k(0)=\max_{\mathbb R^2}v_k,
\]
we have
\[
0\le\psi_k\le1,
\qquad
\psi_k(0)=1.
\]
Hence $y=0$ is the maximum point of $\psi_k$, and therefore
\[
\Delta\psi_k(0)\le0,
\qquad
-\Delta\psi_k(0)\ge0.
\]

Evaluating the equation at $y=0$, and noting that
\[
\psi_k(0)=1,
\qquad
e^{\gamma_k^2(\psi_k(0)^2-1)}=1,
\]
we obtain
\[
0
\le
\frac4{\gamma_k^2}
\left(
1-\frac{\lambda_k}{4\pi}
e^{-\gamma_k^2}
\right)
-
\frac{E_k}{\pi\gamma_k^2}
e^{-\gamma_k^2}.
\]
Multiplying by $\gamma_k^2$ gives
\[
0
\le
4
-
\frac{\lambda_k}{\pi}
e^{-\gamma_k^2}
-
\frac{E_k}{\pi}
e^{-\gamma_k^2}.
\]
Since $E_k>0$, it follows immediately that
\[
\frac{\lambda_k}{\pi}
e^{-\gamma_k^2}
\le
4.
\]
Therefore,
\[
\lambda_k e^{-\gamma_k^2}=O(1).
\]

Now, we prove \eqref{23}. Multiplying (\ref{22}) by $\gamma_{k}^{2}\psi_{k}$, integrating on $\Omega
_{k}$, we have%
 \begin{align*} C_{k}&:=\frac{4\pi}{E_{k}}\int_{B_{1}}\left(  v_{k}^{2}e^{v_{k}^{2}}-\frac
{\lambda_{k}}{4\pi}|v_{k}|^{2}\right)  dx\\&=\int_{B_1}\left(
|\nabla v_{k}|^{2} + |v_{k}|^{2}\right)  dx-\int_{\partial B_1}\frac{\partial v_{k}}{\partial r }  v_{k}d\sigma,  \end{align*} then we have%
\begin{equation}
\frac{\lambda_{k}}{\pi}e^{-\gamma_{k}^{2}}=\frac{4\int_{\Omega_{k}}\psi
_{k}^{2}e^{\gamma_{k}^{2}\left(  \psi_{k}^{2}-1\right)  }dx-C_{k}}%
{\int_{\Omega_{k}}\psi_{k}^{2}dx}\text{.} \label{24}%
\end{equation}
Since  $\lambda_k e^{-\gamma_{k}^{2}}=O(1)$, by the elliptic regularity theory, we have $\psi_{k}\rightarrow
1$, as $k\rightarrow\infty$, and then  
\begin{equation}\label{add772}
    \underset{k\rightarrow\infty}{\lim}\int_{\Omega_{k}}\psi_{k}^{2}%
dx\geq\underset{R\rightarrow\infty}{\lim}\underset{k\rightarrow\infty}{\lim
}\int_{B_{R}}\psi_{k}^{2}dx=\underset{R\rightarrow\infty}{\lim}\left\vert
B_{R}\right\vert =\infty\text{.}
\end{equation}

 On the other hand, using Lemma \ref{17} and (\ref{20}), we have%
\begin{align}
4\int_{\Omega_{k}}\psi_{k}^{2}e^{\gamma_{k}^{2}\left(  \psi_{k}^{2}-1\right)
}dx-C_{k}  &  =\frac{4\pi}{E_{k}}\int_{B_{1}}v_{k}^{2}e^{v_{k}^{2}}%
dx-\frac{4\pi}{E_{k}}\int_{B_{1}}\left(  v_{k}^{2}e^{v_{k}^{2}}-\frac
{\lambda_{k}}{4\pi}|v_{k}|^{2}\right)  dx\nonumber\\
&  =\frac{\lambda_{k}}{E_{k}}\int_{B_{1}}|v_{k}|^{2}dx=o_{k}\left(  1\right)
\text{.}\label{add31}
\end{align}

Combining \eqref{24},\eqref{add772} and \eqref{add31}, we get
\[
\frac{\lambda_k}{\pi}
e^{-\gamma_k^2}
=
\frac{
o_k(1)
}{
\displaystyle\int_{\Omega_k}\psi_k^2\,dy
}
\rightarrow0,
\]
as $k\rightarrow\infty$, and then (\ref{23}) is proved. This completes the proof. \end{proof}

Now, by Proposition \ref{prop:Ek-growth} and Lemma \ref{lem:lambda-bound}, and applying the elliptic regularity theory to \eqref{21} (see Chen and Li \cite{Chen-li}), we have%
\begin{equation}
\phi_{k}\rightarrow\phi_{\infty}=-\log\left(  1+\left\vert x\right\vert
^{2}\right)  \text{ in }C_{loc}^{2}\left(  \mathbb{R}^{2}\right)  \text{, }
\label{01}%
\end{equation}
where $\phi_{\infty}$ satisfies%
\begin{equation}
-\Delta\phi_{\infty}=4e^{2\phi_{\infty}}\text{ in }\mathbb{R}^{2}\text{.}
\label{74}%
\end{equation}
\vskip0.1cm

For any constant $A>1$, we define the truncated function
\begin{equation}
v_{k,A}:=\min\{\frac{\gamma_{k}}{A},v_{k}\} \label{68}%
\end{equation}

\begin{lemma}
\label{25}We have%
\[
\underset{k\rightarrow\infty}{\limsup}\int_{\mathbb{R}^{2}}\left(  |v_{k,A}|%
^{2}+\left\vert \nabla v_{k,A}\right\vert ^{2}\right)  dx\leq\frac{4\pi}%
{A}\text{.}%
\]

\end{lemma}

\begin{proof}
Note that $\left\vert \{x\mid v_{k}\geq\frac{\gamma_{k}}{A}\}\right\vert
\left\vert \frac{\gamma_{k}}{A}\right\vert ^{2}\leq\int_{\{v_{k}\geq
\frac{\gamma_{k}}{A}\}}|v_{k}|^{2}dx\leq1$, then we can find a sequence
$\rho_{k}\rightarrow0$ such that
\[
\{x\mid v_{k}\geq\frac{\gamma_{k}}{A}\}\subset B_{\rho_{k}}\text{.}%
\]
It is clearly that $v_{k}\rightarrow0$ in $L^{p}\left(  B_{1}\right)  $ for
any $p>1$, we have
\[
\underset{k\rightarrow\infty}{\lim}\int_{\{v_{k}>\frac{\gamma_{k}}{A}%
\}}|v_{k,A}|^{p}dx\leq\underset{k\rightarrow\infty}{\lim}\int_{\{v_{k}%
>\frac{\gamma_{k}}{A}\}}|v_{k}|^{p}dx=0\text{,}%
\]
which implies that
\[
\underset{k\rightarrow\infty}{\lim}\int_{\mathbb{R}^{2}}\left(  v_{k}%
-\frac{\gamma_{k}}{A}\right)  ^{+}|v_{k}|^{p}dx=0\text{,}%
\]
for any $p>0$.

Testing equation (\ref{14}) by $\left(  v_{k}-\frac{\gamma_{k}}{A}\right)
^{+}$, and by Lemma \ref{17} and (\ref{20}), for any $R>0$, we have%
\begin{align*}
&  \int_{\mathbb{R}^{2}}\left(  \left\vert \nabla\left(  v_{k}-\frac
{\gamma_{k}}{A}\right)  ^{+}\right\vert ^{2}+\left(  v_{k}-\frac{\gamma_{k}%
}{A}\right)  ^{+}v_{k}\right)  dx\\
&  =\frac{4\pi}{E_{k}}\int_{\mathbb{R}^{2}}v_{k}e^{v_{k}^{2}}\left(
v_{k}-\frac{\gamma_{k}}{A}\right)  ^{+}dx-\frac{\lambda_{k}}{E_{k}}%
\int_{\mathbb{R}^{2}}\left(  v_{k}-\frac{\gamma_{k}}{A}\right)  ^{+}v_{k}dx\\
&  =\frac{4\pi}{E_{k}}\int_{\mathbb{R}^{2}}v_{k}e^{v_{k}^{2}}\left(
v_{k}-\frac{\gamma_{k}}{A}\right)  ^{+}dx+o_{k}\left(  1\right)  \\
&  \geq\frac{4\pi}{E_{k}}\int_{B_{Rr_{k}}}v_{k}e^{v_{k}^{2}}\left(
v_{k}-\frac{\gamma_{k}}{A}\right)  ^{+}dx+o_{k}\left(  1\right) \\
&  =4\frac{A-1}{A}\int_{B_{R}}e^{2\phi_{\infty}(y)+o_{k}\left(  1\right)
}dy+o_{k}\left(  1\right) \text{.}%
\end{align*}
Hence, we can obtain
\[
\underset{k\rightarrow\infty}{\lim\inf}\int_{\mathbb{R}^{2}}\left(  \left\vert
\nabla\left(  v_{k}-\frac{\gamma_{k}}{A}\right)  ^{+}\right\vert ^{2}+\left(
v_{k}-\frac{\gamma_{k}}{A}\right)  ^{+}v_{k}\right)  dx\geq4\frac{A-1}{A}%
\int_{B_{R}}e^{2\phi_{\infty}(y)}dy\text{.}%
\]
Letting $R\rightarrow\infty$, we have%
\[
\underset{k\rightarrow\infty}{\lim\inf}\int_{\mathbb{R}^{2}}\left(  \left\vert
\nabla\left(  v_{k}-\frac{\gamma_{k}}{A}\right)  ^{+}\right\vert ^{2}+\left(
v_{k}-\frac{\gamma_{k}}{A}\right)  ^{+}v_{k}\right)  dx\geq4\pi\frac{A-1}%
{A}\text{.}%
\]
Observe that%
\begin{align*}
&  \int_{\mathbb{R}^{2}}\left(  |v_{k,A}|^{2}+\left\vert \nabla v_{k,A}%
\right\vert ^{2}\right)  dx\\
&  =4\pi-\int_{\mathbb{R}^{2}}\left(  \left\vert \nabla\left(  v_{k}%
-\frac{\gamma_{k}}{A}\right)  ^{+}\right\vert ^{2}+\left(  v_{k}-\frac
{\gamma_{k}}{A}\right)  ^{+}v_{k}\right)  dx\\
&  +\int_{\mathbb{R}^{2}}\left(  v_{k}-\frac{\gamma_{k}}{A}\right)  ^{+}%
v_{k}dx-\int_{\{v_{k}>\frac{\gamma_{k}}{A}\}}|v_{k}|^{2}dx+\int_{\{v_{k}%
>\frac{\gamma_{k}}{A}\}}|v_{k,A}|^{2}dx\\
&  \leq\frac{4\pi}{A}+o_{k}\left(  1\right)  \text{.}%
\end{align*}
Hence, the proof is finished.
\end{proof}

We now refine the estimate of $E_k$ in Lemma \ref{17}.

\begin{proposition}
\label{27}The following equality holds:%
\[
\underset{k\rightarrow\infty}{\lim}\frac{E_{k}}{\gamma_{k}^{2}}=e\pi\text{.}%
\]

\end{proposition}

\begin{proof}
First, we claim that
\begin{equation}
\underset{k\rightarrow\infty}{\lim}\int_{\mathbb{R}^{2}}\left(  e^{v_{k}^{2}%
}-1\right)  dx=\underset{k\rightarrow\infty}{\lim}\frac{E_{k}}{\gamma_{k}^{2}%
}\text{.} \label{218}%
\end{equation}
For any $A>1$, we have
\begin{align}
\int_{\mathbb{R}^{2}}\left(  e^{v_{k}^{2}%
}-1\right)  dx  &  =\int_{\{v_{k}%
<\frac{\gamma_{k}}{A}\}}\left(  e^{v_{k}^{2}}-1\right)  dx+\int_{\{v_{k}\geq\frac{\gamma_{k}}{A}\}}\left(
e^{v_{k}^{2}}-1\right)  dx\nonumber\\
&  \leq\int_{\mathbb{R}^{2}}\left(
e^{v_{k,A}^{2}}-1\right)  dx+\frac{A^{2}%
}{\gamma_{k}^{2}}\int_{\mathbb{R}^{2}}v_{k}^{2}e^{v_{k}^{2}}dx\label{219}\\
&  =\int_{\mathbb{R}^{2}}\left(
e^{v_{k,A}^{2}}-1\right)  dx+\frac{A^{2}%
}{\gamma_{k}^{2}}\left(  E_{k}+\frac{\lambda_{k}}{4\pi}\int_{\mathbb{R}^{2}%
}|v_{k}|^{2}dx\right) \nonumber\\
&  =\int_{\mathbb{R}^{2}}\left(
e^{v_{k,A}^{2}}-1\right)  dx+\frac{A^{2}%
}{\gamma_{k}^{2}}E_{k}+o_k(1)\text{,}\nonumber
\end{align}
where in the last equality, we have used \eqref{20}.

Let $L$ be as defined in (\ref{add774}), we have
\begin{equation}
\underset{k\rightarrow\infty}{\lim}\int_{\mathbb{R}^{2}\backslash B_{L}%
}\left(  e^{pv_{k,A}^{2}}-1\right)  dx\leq  C_p\underset{k\rightarrow\infty}{\lim
}\int_{\mathbb{R}^{2}\backslash B_{L}}|v_{k}|^{2}dx=0 \label{70}%
\end{equation}
for any $p>0$.
\vskip0.1cm

By Lemma \ref{25} and the classical Trudinger-Moser inequality, we have
\[
\underset{k}{\sup}\int_{B_{L}}e^{q\left(  \left(  v_{k,A}-v_{k}\left(
L\right)  \right)  ^{+}\right)  ^{2}}dx<+\infty
\]
for any $q<A$. Since for any $p<q$,
\[
pv_{k,A}^{2}\leq q\left(  \left(  v_{k,A}-v_{k}\left(  L\right)  \right)
^{+}\right)  ^{2}+C\left(  p,q\right)  \text{,}%
\]
we have
\begin{equation}
\underset{k}{\sup}\int_{B_{L}}e^{pv_{k,A}^{2}}dx<+\infty\label{69}%
\end{equation}
for any $1<p<A$. It follows from Vitali's convergence theorem that
\begin{equation}\label{add781}
\underset{k\rightarrow\infty}{\lim}\int_{B_{L}}\left(  e^{v_{k,A}^{2}%
}-1\right)  dx=0\text{.}%
\end{equation}

Combining \eqref{219},  \eqref{70}, \eqref{add781} and \eqref{19}, we have  
\[
\underset{k\rightarrow\infty}{\lim}\int_{\mathbb{R}^{2}}\left(  e^{v_{k}^{2}%
}-1\right)  dx\leq \underset{k\rightarrow\infty}{\liminf}\frac{A^{2}}{\gamma
_{k}^{2}}E_{k}\text{,}%
\]
letting
$A\rightarrow1$, we derive that%
\begin{equation}
\underset{k\rightarrow\infty}{\lim}\int_{\mathbb{R}^{2}}\left(  e^{v_{k}^{2}%
}-1\right)  dx\leq\underset{k\rightarrow\infty}{\liminf}\frac{E_{k}}{\gamma
_{k}^{2}}\text{.}\label{63}%
\end{equation}
It remains to show that%
\begin{equation}
\underset{k\rightarrow\infty}{\lim}\int_{\mathbb{R}^{2}}\left(  e^{v_{k}^{2}%
}-1\right)  dx\geq\underset{k\rightarrow\infty}{\limsup}\frac{E_{k}}{\gamma
_{k}^{2}}\text{.} \label{222}%
\end{equation}
Indeed, for any $R>0$, using \eqref{63} and \eqref{01}, we have
\begin{align}
\int_{\mathbb{R}^{2}}\left(  e^{v_{k}^{2}%
}-1\right)  dx  &  \geq\int_{B_{Rr_{k}}}\left(  e^{v_{k}^{2}}-1\right)
dx\label{03}\\
&   =%
\int_{B_{Rr_{k}}}e^{v_{k}^{2}}dx-
|B_{Rr_k}|\nonumber\\
& =\int_{B_{Rr_{k}}}e^{v_{k}^{2}}dx+o(\frac{E_{k}}{\gamma_{k}^{2}})\nonumber\\
&=\frac{E_{k}}{\gamma_{k}^{2}}\left(\frac{1}{\pi}\int_{B_{R}}e^{2\phi_{\infty}%
}dx+o_k(1)\right)\nonumber\\
&  =\frac{E_{k}}{\gamma_{k}^{2}}\left(1+o_R(1)+o_k(1)\right)%
\text{,}\nonumber
\end{align}
which implies (\ref{222}) holds true. Combining (\ref{63}) and (\ref{222}), we
conclude that (\ref{218}) holds. Finally, by (\ref{19}) and (\ref{218}), the proof is finished.
\end{proof}

\begin{lemma}
\label{61}For any $\eta\in C_0^{\infty}\left(  \mathbb{R}^{2}\right)  $, we have%
\[
\underset{k\rightarrow\infty}{\lim}\frac{1}{E_{k}}\int_{\mathbb{R}^{2}}%
\gamma_{k}v_{k}e^{v_{k}^{2}}\eta dx=\eta\left(  0\right)  \text{.}%
\]

\end{lemma}

\begin{proof}
Suppose supp$\eta\subset B_{\rho}$, we divide the integral into three parts,%
\begin{align*}
\frac{1}{E_{k}}\int_{B_{\rho}}\gamma_{k}v_{k}e^{v_{k}^{2}}\eta dx&  
=\frac{1}{E_{k}}\int_{\{v_{k}>\frac{\gamma_{k}}{A}\}\backslash B_{Rr_{k}}%
}\gamma_{k}v_{k}e^{v_{k}^{2}}\eta dx+\frac{1}{E_{k}}\int_{B_{Rr_{k}}}%
\gamma_{k}v_{k}e^{v_{k}^{2}}\eta dx\\
&  +\frac{1}{E_{k}}\int_{\{v_{k}\leq\frac{\gamma_{k}}{A}\}\backslash
B_{Rr_{k}}}\gamma_{k}v_{k}e^{v_{k}^{2}}\eta dx\\
&  :=I_{1}^{k,R}+I_{2}^{k,R}+I_{3}^{k,R}\text{.}%
\end{align*}
For $I_{1}^{k.R}$, using \eqref{20} and Lemma \ref{17}, we arrive that
\begin{align*}
I_{1}^{k,R}  &  \leq A\left\vert \left\vert \eta\right\vert \right\vert
_{L^{\infty}}\int_{\mathbb{R}^{2}\backslash B_{Rr_{k}}}\frac{1}{E_{k}}%
v_{k}^{2}e^{v_{k}^{2}}dx\\
&  =A\left\vert \left\vert \eta\right\vert \right\vert _{L^{\infty}}\left(
1-\frac{1}{\pi}\int_{B_{R}}e^{2\phi_{\infty}}dy+o_{k}\left(  1\right)
\right)  \text{.}%
\end{align*}
Thus, $\underset{R\rightarrow\infty}{\lim}\underset{k\rightarrow\infty}{\lim
}I_{1}^{k,R}=0$. A careful calculation for $I_{2}^{k,R}$ gives%
\begin{align*}
I_{2}^{k,R}  &  =\frac{1}{E_{k}}\int_{B_{Rr_{k}}}\gamma_{k}v_{k}e^{v_{k}^{2}%
}\eta dx\\
&  =\int_{B_{R}}\frac{1}{\pi}e^{2\phi_{\infty}+o_{k}\left(  1\right)  }%
\eta\left(  r_{k}y\right)  dy\\
&  =\eta\left(  0\right)  \int_{B_{R}}\frac{1}{\pi}e^{2\phi_{\infty}}%
dy+o_{k}\left(  1\right)  \text{,}%
\end{align*}
which implies $\underset{R\rightarrow\infty}{\lim}\underset{k\rightarrow
\infty}{\lim}I_{2}^{k,R}=\eta\left(  0\right)  $. From (\ref{69}) and
(\ref{70}), we have%
\[
\int_{\mathbb{R}^{2}}\left(  e^{pv_{k,A}^{2}}-1\right)  dx<C
\]
for any $1<p<A$. Hence, by Proposition \ref{27}, we can get%
\begin{align*}
I_{3}^{k,R}  &  =\frac{\gamma_{k}}{E_{k}}\int_{\{v_{k}\leq\frac{\gamma_{k}}%
{A}\}\backslash B_{Rr_{k}}}v_{k}\eta\left(  e^{v_{k}^{2}}-1\right)
dx+\frac{\gamma_{k}}{E_{k}}\int_{\{v_{k}\leq\frac{\gamma_{k}}{A}\}\backslash
B_{Rr_{k}}}v_{k}\eta dx\\
&  \leq\left\vert \left\vert \eta\right\vert \right\vert _{L^{\infty}}%
\frac{\gamma_{k}}{E_{k}}\left\vert \left\vert v_{k}\right\vert \right\vert
_{L^{q}\left(  \mathbb{R}^{2}\right)  }\left(  \int_{\mathbb{R}^{2}}\left(
e^{pv_{k,A}^{2}}-1\right)  dx\right)  ^{\frac{1}{p}}+\frac{\gamma_{k}}{E_{k}%
}\left\vert \left\vert \eta\right\vert \right\vert _{L^{2}\left(  \mathbb{R}^{2}\right)  }\left\vert \left\vert v_k\right\vert \right\vert _{L^{2}\left(  \mathbb{R}^{2}\right)  }\rightarrow0,
\end{align*}
as $k\rightarrow\infty$, where $\frac{1}{p}+\frac{1}{q}=1$, with $1<p<A$ and $q>2$.
Then Lemma \ref{61} follows immediately from the estimates for $I_{1}^{k,R}$,
$I_{2}^{k,R}$ and $I_{3}^{k,R}$.
\end{proof}

\begin{proposition}
\label{26}For any $q\in\lbrack1,\infty)$ and $R>0$, it holds that
\begin{equation}
\int_{B_{R}}\left( \gamma_{k}v_{k}\right)  ^{q}dx\rightarrow\int_{B_{R}}%
G^{q}dx, \label{30}%
\end{equation}
as $k\rightarrow\infty$, where $G$ is the Green function of the operator $-\Delta+1  $ on $\mathbb{R}^{2}$ with singularity at $0$. Moreover, we have $\gamma
_{k}v_{k}\rightarrow G$ in $C_{loc}^{1}\left(  B_{R}\backslash\{0\}\right)  $.
\end{proposition}

\begin{proof} For any $1<q<\infty$, there is some $1<\gamma<2$ such that $q<\frac{2\gamma}{2-\gamma}$. 
It follows from (\ref{14}) that%
\begin{equation}
-\Delta\left(  \gamma_{k}v_{k}\right)  +\gamma_{k}v_{k}=\frac{4\pi}{E_{k}%
}\left[   \gamma_{k}v_{k}  e^{v_{k}^{2}}-\frac{\lambda_{k}\gamma_{k}v_{k} }%
{4\pi}  \right]  \text{.} \label{1}%
\end{equation}
 Similar to the proof in \cite[Proposition 3.7]{liruf}, applying \eqref{20} and the regularity theory  we can obtain
\[
\int_{B_R} |\nabla (\gamma_k v_k)|^\gamma dx \leq C(\gamma),
\]
for any $R>0$ and $\gamma\in\left(  1,2\right)  $,   here $C\left(  \gamma\right)  $ is independent of $k$.   The second term of the RHS is not in Li-Ruf situation. Thus, after passing to
a subsequence, we have
\begin{equation}
\gamma_{k}v_{k}\rightharpoonup\omega\text{ in }W^{1,\gamma}\left(
B_{R}\right)  \text{,} \label{62}%
\end{equation}
for some $\omega\in W^{1,\gamma}\left(
B_{R}\right)$. By the compactness of the Sobolev embedding, it follows that%
\begin{equation}
\int_{B_{R}}\left(  \gamma_{k}v_{k}\right) ^{q}dx\rightarrow\int_{B_{R}%
}\omega^{q}dx\text{,} \label{02}%
\end{equation}
\newline
for $1\leq q<\infty$.
\vskip0.1cm

 We next claim that $\omega\neq0$. We assume by
contradiction that $\omega=0$. Let $\bar{v}_{k}=\left(  v_{k}-v_{k}\left(
\delta\right)  \right)  ^{+}$ and $\tau_{k}:=\int_{B_{\delta}}\left\vert
\nabla\bar{v}_{k}\right\vert ^{2}dx<4\pi$ for some $\delta>0$. It is easy to
verify that $\bar{v}_{k}\in H_{0}^{1}\left(  B_{\delta}\right)  $. From the
assumption $\omega=0$, we have $v_{k}\left(  \delta\right)  =\frac
{o_{k}\left(  1\right)  }{\gamma_{k}}$ and $\bar{v}_{k}=v_{k}+\frac
{o_{k}\left(  1\right)  }{\gamma_{k}}$ in $B_{\delta}$.  By a direct
computation, we get%
\[
v_{k}^{2}\leq4\pi\frac{\bar{v}_{k}^{2}}{\tau_{k}}+o_{k}\left(  1\right)
\text{,}%
\]
and%
\[
\underset{L\rightarrow\infty}{\lim}\underset{k\rightarrow\infty}{\lim}%
\int_{B_{Lr_{k}}}e^{v_{k}^{2}}dx\leq\underset{k\rightarrow\infty}{\lim}%
\int_{B_{\delta}}e^{4\pi\frac{\bar{v}_{k}^{2}}{\tau_{k}}}dx\text{.}%
\]
for any $L>0$.
We can check that$\frac{\bar{v}_{k}}{\sqrt{\tau_{k}}}$ is a normalized
concentrating sequence in $H_{0}^{1}\left(  B_{\delta}\right)  $, it follows from \cite[Proposition 2.1]{CJZ-SIAM} that
\[
\underset{k\rightarrow\infty}{\lim}\int_{B_{\delta}}e^{4\pi\frac{\bar{v}%
_{k}^{2}}{\tau_{k}}}dx\leq\left\vert B_{\delta}\right\vert \left(  1+e\right)
\text{.}%
\]
Letting $\delta\rightarrow0$, we derive that%
\[
\underset{L\rightarrow\infty}{\lim}\underset{k\rightarrow\infty}{\lim}%
\int_{B_{Lr_{k}}}e^{v_{k}^{2}}dx=0\text{,}%
\]
which is in contradiction to (\ref{03}). Thus $\omega\neq0$, this together
with (\ref{20}) gives%
\[
o_{k}\left(  1\right)  =\lambda_{k}\int_{B_{R}}|v_{k}|^{2}dx=\frac{\lambda_{k}%
}{\gamma_{k}^{2}}\int_{B_{R}}\left\vert  \gamma_{k}v_{k}\right\vert  ^{2}%
dx=\frac{\lambda_{k}}{\gamma_{k}^{2}}\left[  \int_{B_{R}}|\omega|^{2}%
dx+o_{k}\left(  1\right)  \right]  \text{,}%
\]
and we can deduce
\begin{equation}
\frac{\lambda_{k}}{\gamma_{k}^{2}}\rightarrow0\label{28}%
\end{equation}
as $k\rightarrow\infty$. Combining with Proposition \ref{27}, we conclude that
\[
\underset{k\rightarrow\infty}{\lim}\frac{\lambda_{k}}{E_{k}}\int_{B_{R}}%
\gamma_{k}v_{k}dx=0\text{,}%
\]
then the function $\omega$ satisfies the limit equation of (\ref{1}) as
follows%
\begin{equation}
-\Delta\omega+\omega=4\pi\delta_{0}\text{.} \label{06}%
\end{equation}
By (\ref{02}) and (\ref{28}), applying the elliptic regularity theory to
(\ref{1}), we have%
\[
\gamma_{k}v_{k}\rightarrow\omega\text{ in }C_{loc}^{1}\left(  B_{R}%
\backslash\{0\}\right)  \text{.}%
\]
Using a similar argument in the proof of Lemma 3.8 in \cite{liruf}, we obtain
$\omega=G$, where $G$ is the Green function  of the operator $-\Delta+1  $ on $\mathbb{R}^{2}$ with singularity at $0$.
\end{proof}
\begin{remark}
\label{jjj} We claim that the upper bound of the blow up energy coincides with that of the concentrated one. Indeed, from the fact that $\{u_{k}\}$ is an NCS in $H^1(\mathbb{R}^{2})$, we have
\[
\underset{k\rightarrow\infty}{\lim\sup}\int_{\mathbb{R}^{2}}\left(  e^{4\pi
u_{k}^{2}}-1\right)  dx\leq e\pi\text{.}%
\]
Take the supreme over all normalized concentrating sequences on both sides of the above inequality, we get
\begin{equation}
 \pi e^{1+4\pi A}=\underset{\{u_{k}\}_{k}\text{ is NCS}}{\sup}\underset{k\rightarrow\infty
}{\lim\sup}\int_{\mathbb{R}^{2}}\left(  e^{4\pi u_{k}^{2}}-1\right)
dx\leq e\pi
\text{,}\label{143}%
\end{equation}
 where the first equality follows from \cite{Nguyen-JFA} and $A=\underset{r\rightarrow0}{\lim}\left(  G\left(  r\right)  +\frac
{1}{4\pi}\log r^{2}\right)  $.
By Lemma \ref{lem2.4} and (\ref{20}), then
\begin{equation}
e\pi\leq\underset{k\rightarrow\infty}{\lim}S\left(  \lambda_{k},2\right)=\underset{k\rightarrow\infty
}{\lim\sup}\int_{\mathbb{R}^{2}}\left(  e^{4\pi u_{k}^{2}}-1\right)
dx
\leq\pi e^{1+4\pi A}\text{.}\label{145}%
\end{equation}
Combining (\ref{143}) and ((\ref{145})), we can get the desired results.
\end{remark}
\bigskip Based on the above analysis, we can give the estimate of polynomial
term $I_{P}$ in (\ref{118}):
\begin{proof}[Proof of Proposition \ref{223}]
According to Proposition \ref{26}, we have
\begin{equation}
\frac{\lambda_{k}}{E_{k}}\int_{\mathbb{R}^{2}}|v_{k}|^{2}dx\geq\frac{\lambda
_{k}}{E_{k}}\int_{B_{R}}|v_{k}|^{2}dx=\frac{\lambda_{k}}{E_{k}\gamma_{k}^{2}%
}\int_{B_{R}}\left\vert
\gamma_{k}v_{k}\right\vert  ^{2}dx=\frac{\lambda_{k}}{E_{k}\gamma_{k}^{2}%
}\int_{B_{R}}G^{2}dx+o_{k}\left(  1\right)  \text{.} \label{54}%
\end{equation}
By Proposition \ref{27}, there exists some constant $C>0$ such that%
\[
\frac{\lambda_{k}}{E_{k}}\int_{\mathbb{R}^{2}}|v_{k}|^{2}dx\geq\frac
{C\lambda_{k}}{\gamma_{k}^{4}},%
\]
which finishes the estimate of Lebesgue integral $I_{P}$.
\end{proof}

\section{Refined Asymptotic Expansion of $u_k$ near the blow up point and  Estimates for the Exponential  term $I_E$}\label{esti Exponential}

\bigskip In this section, our aim is to estimate the integral of the
exponential term $I_{E}$ in \eqref{118}. To this end, we need to study a refined asymptotic expansion of $u_k$ near the blow up point $0$, when $k\rightarrow \infty$.

  Let $$t_{k}\left(  x\right)  =-\phi_{\infty}\left(  \frac{x}{r_{k}}\right)
=\log\left(  1+\left(  \frac{x}{r_{k}}\right)  ^{2}\right),$$ and  $r_{k,\delta}>0$ be such that
\begin{equation}
t_{k}\left(  r_{k,\delta}\right)  =\log\left(  1+\left(  \frac{r_{k,\delta}%
}{r_{k}}\right)  ^{2}\right)  =\delta\gamma_{k}^{2}\text{ for some }\delta
\in\left(  0,1\right)  \text{.}\label{08}%
\end{equation}
   From (\ref{08}), we have the following estimates for $r_{k,\delta}$:
\begin{equation}
r_{k,\delta}^{2}=r_{k}^{2}\exp\left(  \delta\gamma_{k}^{2}+o_{k}\left(
1\right)  \right)  \label{127}%
\end{equation}
and
\[
\frac{r_{k,\delta}}{r_{k}}\rightarrow\infty\text{ as }k\rightarrow
\infty\text{.}%
\]
Let $L$ be defined in (\ref{add774}) and we rewrite $I_{E}$  as follows:%
\[I_{E}
=\frac{4\pi}{E_{k}}\left(\int_{\mathbb{R}^{2}\backslash B_{L}}+ \int_{B_{L}\backslash B_{r_{k,\delta}}} +\int_{B_{r_{k,\delta}}}\right) v_{k}%
^{2}e^{v_{k}^{2}}dx=I_{1}+I_{2}+I_{3}.%
\]%
We estimate $I_{1}$ as follows%
\begin{align}\label{84}
I_{1} \leq\frac{C}{E_{k}}\int_{\mathbb{R}^{2}\backslash B_{L}}|v_{k}|%
^{2}dx\nonumber& =\frac{C}{E_{k}\gamma_{k}^{2}}\int_{\mathbb{R}^{2}\backslash B_{L}}\left\vert
\gamma_{k}v_{k}\right\vert  ^{2}dx
 \\ &\leq\frac{C}{E_{k}\gamma_{k}^{2}}\int_{\mathbb{R}^{2}\backslash B_{L}%
}\left[  \left\vert \nabla\left(  \gamma_{k}v_{k}\right)  \right\vert
^{2}+\left\vert  \gamma_{k}v_{k}\right\vert  ^{2}\right]  dx\text{.}
\end{align}
By equation (\ref{1}), we get
\begin{align}
\int_{\mathbb{R}^{2}\backslash B_{L}}\left[  \left\vert \nabla\left(
\gamma_{k}v_{k}\right)  \right\vert ^{2}+\left\vert  \gamma_{k}v_{k}\right\vert
^{2}\right]  dx  &  =\frac{4\pi}{E_{k}}\int_{\mathbb{R}^{2}\backslash
B_{L}}\left[  \left\vert  \gamma_{k}v_{k}\right\vert  ^{2}e^{v_{k}^{2}}-\frac
{\lambda_{k}}{4\pi}\left\vert  \gamma_{k}v_{k}\right\vert  ^{2}\right]  dx\nonumber\\
&  -\int_{\partial B_{L}}\frac{\partial\left(  \gamma_{k}v_{k}\right)
}{\partial r}\left(  \gamma_{k}v_{k}\right)  d\sigma \text{.} \label{82}%
\end{align}
From Proposition \ref{27} and \eqref{concentrat}, we derive%
\begin{equation}
\underset{k\rightarrow\infty}{\lim}\frac{4\pi}{E_{k}}\int_{\mathbb{R}%
^{2}\backslash B_{L}}\left\vert  \gamma_{k}v_{k}\right\vert  ^{2}e^{v_{k}^{2}}dx\leq
C\underset{k\rightarrow\infty}{\lim}\int_{\mathbb{R}^{2}\backslash B_{L}}%
|v_{k}|^{2}dx=0\text{,} \label{04}%
\end{equation}
and using Proposition \ref{27}  and (\ref{20}), we obtain%
\begin{equation}
\underset{k\rightarrow\infty}{\lim}\frac{\lambda_{k}}{E_{k}}\int
_{\mathbb{R}^{2}\backslash B_{L}}\left\vert  \gamma_{k}v_{k}\right\vert
^{2}dx=0\text{.} \label{05}%
\end{equation}
Then by (\ref{82}), (\ref{04}) (\ref{05}) and (\ref{06}), we get%
\begin{align*}
\underset{k\rightarrow\infty}{\lim}\int_{\mathbb{R}^{2}\backslash B_{L}%
}\left[  \left\vert \nabla\left(  \gamma_{k}v_{k}\right)  \right\vert
^{2}+\left\vert \gamma_{k}v_{k}\right\vert  ^{2}\right]  dx  &  \leq-\underset
{k\rightarrow\infty}{\lim}\int_{\partial B_{L}}\frac{\partial\left(
\gamma_{k}v_{k}\right)  }{\partial r}\left(  \gamma_{k}v_{k}\right)  d\sigma\\
&  =-G\left(  L\right)  \int_{\partial B_{L}}\frac{\partial G}{\partial
r}d\sigma\\
&  =G\left(  L\right)  \left(  4\pi-\int_{B_{L}}Gdx\right)  \text{. }%
\end{align*}
Thus, we can obtain that $I_{1}=O\left(  \frac{1}{\gamma_{k}^{4}}\right)  $
from Proposition \ref{27}.

In order to estimate $I_{2}$ and $I_{3}$, we will make a Taylor type expansion of
$v_{k}$. Let $S_{k}$ be given by
\[
S_{k}\left(  z\right)  =S_{0}\left(  \frac{z}{r_{k}}\right)  \text{,}%
\]
where $S_{0}$ is the radial solution around $0\in\mathbb{R}^{2}$ of%
\begin{equation}
-\Delta S_{0}-8\exp\left(  2\phi_{\infty}\right)  S_{0}=4\exp\left(
2\phi_{\infty}\right)  \left(  \phi_{\infty}^{2}+\phi_{\infty}\right)  \text{
in }\mathbb{R}^{2} \label{07}%
\end{equation}
with $S_{0}\left(  0\right)  =0$.
From \cite{Malchiodi}, we know that $S_{0}$ has the following explicit formula:%
\begin{equation}
S_{0}\left(  r\right)  =\phi_{\infty}\left(  r\right)  +\frac{2r^{2}}{1+r^{2}%
}-\frac{1}{2}\phi_{\infty}^{2}\left(  r\right)  +\frac{1-r^{2}}{1+r^{2}}%
\int_{1}^{1+r^{2}}\frac{\log t}{1-t}dt\text{.}\label{yy}%
\end{equation}
In particular%
\begin{equation}
S_{0}\left(  r\right)  =\frac{A_{0}}{4\pi}\log\frac{1}{r^{2}}+B_{0}+O\left(
r^{-2}\log r^{2}\right)  \text{ as }r\rightarrow\infty\text{,} \label{s0}%
\end{equation}
with
\begin{equation}
A_{0}=\int_{\mathbb{R}^{2}}-\Delta S_{0}dy=4\pi\text{ and }B_{0}=\frac{\pi
^{2}}{6}+2\text{.}\label{128}%
\end{equation}

\begin{lemma}
\label{31}It holds%
\[
v_{k}\left(  r\right)  =\gamma_{k}-\frac{t_{k}\left(  r\right)  }{\gamma_{k}%
}+\frac{S_{k}\left(  r\right)  }{\gamma_{k}^{3}}+O\left(  \frac{1+t_{k}\left(
r\right)  }{\gamma_{k}^{5}}\right)
\]
for $r\in\left[  0,r_{k,\delta}\right]  $ and $r=\left\vert x\right\vert $.
\end{lemma}

\begin{proof}
Let $\xi_{k}$ be given by
\begin{equation}
v_{k}=\gamma_{k}-\frac{t_{k}}{\gamma_{k}}+\frac{\xi_{k}}{\gamma_{k}^{3}%
}\text{,} \label{B}%
\end{equation}
and $\rho_{k}$ be defined as%
\begin{equation}
\rho_{k}=\sup\left\{  r\in(0,r_{k,\delta}]:\left\vert S_{k}-\xi_{k}\right\vert
\leq1+t_{k}\text{ in }[0,r]\right\}  \text{.} \label{P}%
\end{equation}
Notice that (\ref{s0}) and (\ref{P}) imply $\xi_{k}=O\left(  1+t_{k}\right)  $
in $B_{\rho_{k}}\left(  0\right)  $. In particular, from (\ref{B}) we get%
\begin{equation}
v_{k}=\gamma_{k}-\frac{t_{k}}{\gamma_{k}}+\frac{O\left(  1+t_{k}\right)
}{\gamma_{k}^{3}}\text{,} \label{extend}%
\end{equation}
and
\begin{equation}
v_{k}^{2}=\gamma_{k}^{2}-2t_{k}+\frac{t_{k}^{2}+2\xi_{k}}{\gamma_{k}^{2}%
}+O\left(  \frac{1+t_{k}^{2}}{\gamma_{k}^{4}}\right)  \label{squel}%
\end{equation}
in $B_{\rho_{k}}\left(  0\right)  $. Then, we have%
\begin{equation}
\exp\left(  \frac{t_{k}^{2}+2\xi_{k}}{\gamma_{k}^{2}}+O\left(  \frac
{1+t_{k}^{2}}{\gamma_{k}^{4}}\right)  \right)  =1+\frac{t_{k}^{2}+2\xi_{k}%
}{\gamma_{k}^{2}}+O\left(  \frac{\left(  1+t_{k}^{4}\right)  \exp\left(
t_{k}^{2}/\gamma_{k}^{2}\right)  }{\gamma_{k}^{4}}\right)  \label{exp1}%
\end{equation}
in $B_{\rho_{k}}\left(  0\right)  $. Using (\ref{3.5}), (\ref{extend}),
(\ref{squel}) and (\ref{exp1}), we have%
\begin{equation}
\frac{4\pi}{E_{k}}v_{k}\exp\left(  v_{k}^{2}\right)  =\frac{4\exp\left(
-2t_{k}\right)  }{r_{k}^{2}\gamma_{k}}\left[  1+\frac{2\xi_{k}+t_{k}^{2}%
-t_{k}}{\gamma_{k}^{2}}+O\left(  \frac{\left(  1+t_{k}^{4}\right)  \exp\left(
t_{k}^{2}/\gamma_{k}^{2}\right)  }{\gamma_{k}^{4}}\right)  \right]
\label{main}%
\end{equation}
in $B_{\rho_{k}}\left(  0\right)  $. Now, we claim that
\begin{equation}
\frac{\lambda_{k}}{E_{k}}v_{k}+v_{k}=o\left(  \frac{\exp\left(  -2t_{k}%
\right)  \exp\left(  t_{k}^{2}/\gamma_{k}^{2}\right)  }{\gamma_{k}^{5}%
r_{k}^{2}}\right)  \text{.} \label{o}%
\end{equation}
Indeed,
\[
\frac{\exp\left(  t_{k}\left(  -2+t_{k}/\gamma_{k}^{2}\right)  \right)
}{r_{k}^{2}}=\exp\left(  \gamma_{k}-\frac{t_{k}}{\gamma_{k}}\right)  ^{2}%
\frac{\pi}{E_{k}}\gamma_{k}^{2}\text{.}%
\]
Since $t_{k}\left(  x\right)  \leq\delta\gamma_{k}^{2}$ for any $x\in
B_{r_{k,\delta}}\left(  0\right)  $, it can be deduced from Proposition \ref{27}, (\ref{28}) and
(\ref{extend}) that
\[
\frac{\left(  \frac{\lambda_{k}}{E_{k}}v_{k}+v_{k}\right)  \gamma_{k}^{5}%
}{\frac{\exp\left(  t_{k}\left(  -2+t_{k}/\gamma_{k}^{2}\right)  \right)
}{r_{k}^{2}}}=\frac{\left(  \gamma_{k}-\frac{t_{k}}{\gamma_{k}}+O\left(
\frac{1+t_{k}}{\gamma_{k}^{3}}\right)  \right)  \gamma_{k}^{5}\left(
1+\frac{\lambda_{k}}{E_{k}}\right)  }{\exp\left(  \gamma_{k}-\frac{t_{k}%
}{\gamma_{k}}\right)  ^{2}\frac{\pi}{E_{k}}\gamma_{k}^{2}}\leq\frac
{C\gamma_{k}^{6}}{\exp\left(  \left(  1-\delta\right)  ^{2}\gamma_{k}%
^{2}\right)  }\rightarrow0\text{,}%
\]
as $k\rightarrow\infty$. Thus, the claim (\ref{o}) is proved.
\vskip0.1cm

Combining (\ref{o}), (\ref{14}), (\ref{B}) and (\ref{main}), we get%
\begin{equation}
-\Delta\xi_{k}=\frac{4\exp\left(  -2t_{k}\right)  }{r_{k}^{2}}\left(  2\xi
_{k}+t_{k}^{2}-t_{k}+O\left(  \frac{\left(  1+t_{k}^{4}\right)  \exp\left(
t_{k}^{2}/\gamma_{k}^{2}\right)  }{\gamma_{k}^{2}}\right)  \right)
\label{get}%
\end{equation}
in $B_{\rho_{k}}\left(  0\right)  $. Next, we estimate the function $\left\vert \xi_{k}-S_{k}\right\vert $. From
(\ref{07}) and (\ref{get}), we have%
\begin{equation}
-\Delta\left(  \xi_{k}-S_{k}\right)  =\frac{8\exp\left(  -2t_{k}\right)
}{r_{k}^{2}}\left[  \left(  \xi_{k}-S_{k}\right)  +O\left(  \frac{\left(
1+t_{k}^{4}\right)  \exp\left(  t_{k}^{2}/\gamma_{k}^{2}\right)  }{\gamma
_{k}^{2}}\right)  \right]  \label{505}%
\end{equation}
for all $0\leq r\leq$ $\rho_{k}$. Observe that
\begin{equation}
\int_{B_{r}\left(  0\right)  }\left(  -\Delta\left(  \xi_{k}-S_{k}\right)
\right)  dy=-2\pi r\left(  \xi_{k}-S_{k}\right)  ^{\prime}\left(  r\right)
\text{.} \label{observe}%
\end{equation}
Now, we estimate the integral of the functions on the right-hand side of the
equations (\ref{505}) over the balls $B_{\rho_{k}}\left(  0\right)  $. Since
$2-\frac{t_{k}}{\gamma_{k}^{2}}\geq2-\delta>1$ by (\ref{08}), there exists
some $\alpha>1$ and $C>0$ such that
\begin{equation}
\left(  1+t_{k}^{4}\right)  \exp\left(  t_{k}\left(  -2+t_{k}/\gamma_{k}%
^{2}\right)  \right)  \leq C\exp\left(  -\alpha t_{k}\right)  \label{T}%
\end{equation}
in $B_{\rho_{k}}\left(  0\right)  $. Using (\ref{T}), we can get%
\begin{equation}
\int_{B_{r}\left(  0\right)  }\frac{8\left(  1+t_{k}^{4}\right)  \exp\left(
t_{k}\left(  -2+t_{k}/\gamma_{k}^{2}\right)  \right)  }{r_{k}^{2}}dy\leq
C_{\alpha}\left(  1-\left(  1+\left(  \frac{r}{r_{k}}\right)  ^{2}\right)
^{1-\alpha}\right)  \text{.}\label{09}%
\end{equation}
It remains to compute the integral of first term in (\ref{505}). Given that
$$
   \left\vert \left(  \xi_{k}-S_{k}\right)  \left(  r\right)  \right\vert
\leq\left\Vert \left(  \xi_{k}-S_{k}\right)  ^{\prime}\right\Vert _{L^{\infty
}\left(  \left[  0,\rho_{k}\right]  \right)  }r, 
$$
this inequality leads to
\begin{equation}
\int_{B_{r}\left(  0\right)  }\frac{8\exp\left(  -2t_{k}\right)  }{r_{k}^{2}%
}\left\vert \xi_{k}-S_{k}\right\vert dy\leq8\pi\left\Vert \left(  \xi
_{k}-S_{k}\right)  ^{\prime}\right\Vert _{L^{\infty}\left(  \left[  0,\rho
_{k}\right]  \right)  }r_{k}\left(  \arctan\left(  \frac{r}{r_{k}}\right)
-\frac{\frac{r}{r_{k}}}{1+\left(  \frac{r}{r_{k}}\right)  ^{2}}\right)
\text{.}\label{506}%
\end{equation}
Combining (\ref{observe}), (\ref{505}), (\ref{09}) with (\ref{506}), there exists a constant $C>1$ such that%
\begin{equation}
\frac{r\left\vert \left(  \xi_{k}-S_{k}\right)  ^{\prime}\left(  r\right)
\right\vert }{C}\leq\frac{\left(  \frac{r}{r_{k}}\right)  ^{2}%
}{\gamma_{k}^{2}\left(  1+\left(  \frac{r}{r_{k}}\right)  ^{2}\right)  }%
+\frac{r_{k}\left\Vert \left(  \xi_{k}-S_{k}\right)  ^{\prime}\right\Vert
_{L^{\infty}\left(  \left[  0,\rho_{k}\right]  \right)  }\left(  \frac
{r}{r_{k}}\right)  ^{3}}{1+\left(  \frac{r}{r_{k}}\right)  ^{3}}\label{120}%
\end{equation}
for all $0\leq r\leq$ $\rho_{k}$. Now we show that%
\begin{equation}
r_{k}\left\Vert \left(  \xi_{k}-S_{k}\right)  ^{\prime}\right\Vert
_{L^{\infty}\left(  \left[  0,\rho_{k}\right]  \right)  }=O\left(  \frac
{1}{\gamma_{k}^{2}}\right)  \text{.}\label{121}%
\end{equation}
We prove this by contradiction and assume that
\begin{equation}
\gamma_{k}^{2}r_{k}\left\Vert \left(  \xi_{k}-S_{k}\right)  ^{\prime
}\right\Vert _{L^{\infty}\left(  \left[  0,\rho_{k}\right]  \right)  }%
=\gamma_{k}^{2}r_{k}\left\vert \left(  \xi_{k}-S_{k}\right)  ^{\prime}\left(
s_{k}\right)  \right\vert \rightarrow\infty,\text{ as }k\rightarrow
\infty\label{122}%
\end{equation}
for some $s_{k}\in\left[  0,\rho_{k}\right]  $. Using (\ref{120}) and
(\ref{122}), we immediately obtain
\begin{equation}
s_{k}=O\left(  r_{k}\right)  \text{ and }r_{k}=O\left(  s_{k}\right)
\text{,}\label{300}%
\end{equation}
this implies that there exists $\alpha_{0}\in(0,+\infty]$ such that
$\frac{\rho_{k}}{r_{k}}\rightarrow$ $\alpha_{0}$ as $k\rightarrow\infty$.
\vskip0.1cm

Let $\tilde{\xi}_{k}$ be given by
\[
\tilde{\xi}_{k}\left(  s\right)  =\frac{\left(  \xi_{k}-S_{k}\right)  \left(
r_{k}s\right)  }{r_{k}\left\Vert \left(  \xi_{k}-S_{k}\right)  ^{\prime
}\right\Vert _{L^{\infty}\left(  \left[  0,\rho_{k}\right]  \right)  }%
}\text{.}%
\]
Then by (\ref{120}) and (\ref{122}), there exists a constant $C>0$ such
that%
\begin{equation}
\left\vert \tilde{\xi}_{k}^{\prime}\left(  s\right)  \right\vert \leq
\frac{C}{1+s}\text{ in }\left[  0,\rho_{k}/r_{k}\right]  \text{.}\label{7}%
\end{equation}
Combining (\ref{505}), (\ref{122}) and (\ref{7}), using the standard elliptic
regularity theory, we get that
\begin{equation}
\tilde{\xi}_{k}\rightarrow\tilde{\xi}\text{ in }C_{loc}^{1}\left(
B_{\alpha_{0}}\left(  0\right)  \right)  \text{ as }k\rightarrow\infty
\text{,}\label{123}%
\end{equation}
for some $\tilde{\xi}\in C_{loc}^{1,\alpha}\left(  B_{\alpha_{0}}\left(  0\right)
\right)  $ satisfying%
\[
\left\{
\begin{array}
[c]{c}%
-\Delta\tilde{\xi}=8\exp\left(  2\phi_{\infty}\right)  \tilde{\xi}\text{ in
}B_{\alpha_{0}}\left(  0\right)  \text{,}\\
\tilde{\xi}\left(  0\right)  =0\text{,}\\
\tilde{\xi}\text{ radially symmetric around }0\in\mathbb{R}^{2}\text{.}%
\end{array}
\right.
\]
From \cite[Lemma C.1.]{PLA}, one can obtain
\begin{equation}
\tilde{\xi}\equiv0\text{ in }B_{\alpha_{0}}\left(  0\right)  \text{.}%
\label{124}%
\end{equation}

Now, we can improve the estimates in (\ref{120}) from (\ref{124}). Indeed,
using (\ref{7}), (\ref{123}), (\ref{124}) and the dominated convergence
theorem, we can obtain%
\begin{equation}
\int_{B_{\rho_{k}}\left(  0\right)  }\frac{\exp\left(  -2t_{k}\right)
}{r_{k}^{2}}\left\vert \xi_{k}-S_{k}\right\vert dy=o\left(  r_{k}\left\Vert
\left(  \xi_{k}-S_{k}\right)  ^{\prime}\right\Vert _{L^{\infty}\left(  \left[
0,\rho_{k}\right]  \right)  }\right)  \text{.}\label{125}%
\end{equation}
Repeating the argument for (\ref{120}), replacing (\ref{506}) with
(\ref{125}), and utilizing (\ref{122}), we get%
\begin{equation}
r\left\vert \left(  \xi_{k}-S_{k}\right)  ^{\prime}\left(  r\right)
\right\vert =o\left(  r_{k}\left\vert \left\vert \left(  \xi_{k}-S_{k}\right)
^{\prime}\right\vert \right\vert _{L^{\infty}\left(  \left[  0,\rho
_{k}\right]  \right)  }\right)  ,\label{126}%
\end{equation}
for all $0\leq r\leq$ $\rho_{k},$ as $k\rightarrow\infty$. If we choose
$r=s_{k}$ in (\ref{126}), then we immediately have $s_{k}=o\left(
r_{k}\right)  $, but this contradicts to (\ref{300}). Hence, (\ref{121}) is proved. Now, plugging (\ref{121}) into (\ref{120}), using the fact $\xi_{k}\left(
0\right)  =S_{k}\left(  0\right)  =0$ and the fundamental theorem of calculus,
we obtain that%
\[
\left\vert \left\vert  \xi_{k}-S_{k} \right\vert \right\vert
_{L^{\infty}\left(  \left[  0,\rho_{k}\right]  \right)  }=O\left(
\frac{1+t_{k}}{\gamma_{k}^{2}}\right)
\]
as $k\rightarrow\infty$, which together with (\ref{P}) yields $\rho
_{k}=r_{k,\delta\text{ }}$and we can conclude the proof.
\end{proof}
  Now, we turn to estimate $I_{2}$. From Lemma \ref{31}, it is clearly
that
\[
v_{k}\left(  r\right)  \geq\left(  1-\delta+o_{k}\left(  1\right)  \right)
\gamma_{k}\text{ for }r\in\left[  0,r_{k,\delta}\right]  \text{.}%
\]
Then we choose some $A>1$ which is defined in (\ref{68}) such that
\[
v_{k}\left(  r\right)  \leq\left(  1-\delta+o_{k}\left(  1\right)  \right)
\gamma_{k}\leq\frac{\gamma_{k}}{A}\text{ for }r\in\left(  r_{k,\delta}%
,\infty\right)  \text{.}%
\]
From Proposition \ref{26} and (\ref{69}), we have
\begin{align}
I_{2}  &  =\frac{4\pi}{E_{k}\gamma_{k}^{2}}\int_{B_{L}\backslash
B_{r_{k,\delta}}}\left\vert \gamma_{k}v_{k}\right\vert  ^{2}e^{v_{k}^{2}%
}dx\nonumber\\
&  \leq\frac{4\pi}{E_{k}\gamma_{k}^{2}}\int_{B_{L}}\left\vert  \gamma_{k}%
v_{k}\right\vert  ^{2}e^{v_{k,A}^{2}}dx\nonumber\\
&  \leq\frac{4\pi}{E_{k}\gamma_{k}^{2}}\left(  \int_{B_{L}}\left\vert  \gamma
_{k}v_{k}\right\vert  ^{2p}dx\right)  ^{\frac{1}{p}}\left(  \int_{B_{L}%
}e^{qv_{k,A}^{2}}dx\right)  ^{\frac{1}{q}}\label{80}\\
&  =O\left(  \frac{1}{\gamma_{k}^{4}}\right)  \text{,}\nonumber
\end{align}
where $\frac{1}{p}+\frac{1}{q}=1$ and $1<q<A$.
\vskip0.1cm

It remains to compute $I_{3}$. In view of (\ref{main}) and the fact that $\rho_k=r_{k,\delta}$, we have%
\begin{equation}
\frac{4\pi}{E_{k}}v_{k}\exp\left(  v_{k}^{2}\right)  =\frac{4\exp\left(
-2t_{k}\right)  }{r_{k}^{2}\gamma_{k}}\left[  1+\frac{2S_{k}+t_{k}^{2}-t_{k}%
}{\gamma_{k}^{2}}+O\left(  \frac{\left(  1+t_{k}^{4}\right)  \exp\left(
t_{k}^{2}/\gamma_{k}^{2}\right)  }{\gamma_{k}^{4}}\right)  \right]  \label{32}%
\end{equation}
in $B_{r_{k,\delta}}\left(  0\right)  $. Thanks to (\ref{T}), we can estimate
the first error term in (\ref{32}). Specifically, we can find $\alpha>1$ such that%
\[
\frac{4\pi}{E_{k}}v_{k}\exp\left(  v_{k}^{2}\right)  =\frac{4\exp\left(
-2t_{k}\right)  }{r_{k}^{2}\gamma_{k}}\left(  1+\frac{2S_{k}+t_{k}^{2}-t_{k}%
}{\gamma_{k}^{2}}\right)  +O\left(  \frac{\exp\left(  -\alpha t_{k}\right)
}{r_{k}^{2}\gamma_{k}^{5}}\right)  \text{.}%
\]
Using Lemma \ref{31}, for $x\in B_{r_{k,\delta}}\left(  0\right)  $, we have
\begin{align}
\frac{4\pi}{E_{k}}v_{k}^{2}\exp\left(  v_{k}^{2}\right)   &  =\frac
{4\exp\left(  -2t_{k}\right)  }{r_{k}^{2}\gamma_{k}}\left(  1+\frac
{2S_{k}+t_{k}^{2}-t_{k}}{\gamma_{k}^{2}}\right)  \left(  \gamma_{k}%
-\frac{t_{k}}{\gamma_{k}}+\frac{S_{k}}{\gamma_{k}^{3}}+O\left(  \frac{1+t_{k}%
}{\gamma_{k}^{5}}\right)  \right)  +O\left(  \frac{\exp\left(  -\alpha
t_{k}\right)  }{r_{k}^{2}\gamma_{k}^{4}}\right)  \nonumber\\
&  =\frac{4\exp\left(  -2t_{k}\right)  }{r_{k}^{2}}\left(  1+\frac
{2S_{k}+t_{k}^{2}-t_{k}}{\gamma_{k}^{2}}\right)  \left(  1-\frac{t_{k}}%
{\gamma_{k}^{2}}+\frac{S_{k}}{\gamma_{k}^{4}}+O\left(  \frac{1+t_{k}}%
{\gamma_{k}^{6}}\right)  \right)  +O\left(  \frac{\exp\left(  -\alpha
t_{k}\right)  }{r_{k}^{2}\gamma_{k}^{4}}\right)  \nonumber\\
&  =\frac{4\exp\left(  -2t_{k}\right)  }{r_{k}^{2}}\left[  1+\frac
{2S_{k}+t_{k}^{2}-2t_{k}}{\gamma_{k}^{2}}+\frac{S_{k}-t_{k}\left(
2S_{k}+t_{k}^{2}-t_{k}\right)  }{\gamma_{k}^{4}}+O\left(  \frac{1+t_{k}^{3}%
}{\gamma_{k}^{6}}\right)  \right]  \label{33}\\
&  \text{ \ \ \ }+O\left(  \frac{\exp\left(  -\alpha t_{k}\right)  }{r_{k}%
^{2}\gamma_{k}^{4}}\right)  \text{.}\nonumber
\end{align}

From the above analysis, we can obtain the integral estimates on the small
ball $B_{r_{k,\delta}}\left(  0\right)  $.

\begin{proposition}
It holds that%
\[
I_{3}=\frac{4\pi}{E_{k}}\int_{B_{r_{k,\delta}}\left(  0\right)  }v_{k}^{2}%
\exp\left(  v_{k}^{2}\right)  dx=4\pi+O\left(  \frac{1}{\gamma_{k}^{4}%
}\right)  \text{.}%
\]

\end{proposition}

\begin{proof}
It follows from (\ref{33}) that
\begin{align*}
&  \frac{4\pi}{E_{k}}\int_{B_{r_{k,\delta}}\left(  0\right)  }v_{k}^{2}%
\exp\left(  v_{k}^{2}\right)  dx\\
&  =\int_{B_{r_{k,\delta}}\left(  0\right)  }\frac{4\exp\left(  -2t_{k}%
\right)  }{r_{k}^{2}}\left(  1+\frac{2S_{k}+t_{k}^{2}-2t_{k}}{\gamma_{k}^{2}%
}\right)  dx+\int_{B_{r_{k,\delta}}\left(  0\right)  }\frac{4\exp\left(
-2t_{k}\right)  }{r_{k}^{2}}\left(  \frac{S_{k}-t_{k}\left(  2S_{k}+t_{k}%
^{2}-t_{k}\right)  }{\gamma_{k}^{4}}\right)  dx\\
&  \text{ \ \ \ }+O\left(  \int_{B_{r_{k,\delta}}\left(  0\right)  }%
\frac{4\exp\left(  -2t_{k}\right)  }{r_{k}^{2}}\frac{1+t_{k}^{3}}{\gamma
_{k}^{6}}dx\right)  +O\left(  \int_{B_{r_{k,\delta}}\left(  0\right)  }%
\frac{\exp\left(  -\alpha t_{k}\right)  }{r_{k}^{2}\gamma_{k}^{4}}dx\right)
\\
&  :=L_{1}+L_{2}+L_{3}+L_{4}\text{.}%
\end{align*}
From (\ref{127}), (\ref{01}), (\ref{s0}) and (\ref{yy}), we have
\begin{align*}
{L_{2}} &  {=}\frac{4}{\gamma_{k}^{4}}\int_{{B_{{{r_{k,\delta}/r}}_{k}}%
}\left(  {0}\right)  }\frac{S_{0}\left(  y\right)  +\phi_{\infty}\left(
y\right)  \left(  2S_{0}\left(  y\right)  +\phi_{\infty}^{2}\left(  y\right)
+\phi_{\infty}\left(  y\right)  \right)  }{\left(  1+\left\vert y\right\vert
^{2}\right)  ^{2}}dy\\
&  =O\left(  \frac{1}{\gamma_{k}^{4}}\right)  \text{.}%
\end{align*}
Since $\alpha>1$, the calculation of $L_{3}$ and $L_{4}$ follow in a similar
manner,%
\[
L{_{3}=}\frac{1}{\gamma_{k}^{6}}\int_{{B_{{{r_{k,\delta}/r}}_{k}}}\left(
{0}\right)  }\frac{1-\phi_{\infty}^{3}\left(  y\right)  }{\left(  1+\left\vert
y\right\vert ^{2}\right)  ^{2}}dy=O\left(  \frac{1}{\gamma_{k}^{6}}\right)
\text{,}%
\]
and
\[
L{_{4}}=\frac{1}{{\gamma_{k}^{4}}}{\int_{{B_{{{r_{k,\delta}/r}}_{k}}}\left(
{0}\right)  }}{\left(  {\frac{1}{{1+}\left\vert x\right\vert {{^{2}}}}%
}\right)  ^{\alpha}}dy=O\left(  {\frac{1}{{\gamma_{k}^{4}}}}\right)  \text{.}%
\]
Hence
\begin{align}
&  \frac{{4\pi}}{{{E_{k}}}}\int_{B_{r_{k,\delta\text{ }}}\left(  x_{k}\right)
}{v_{k}^{2}}\exp\left(  {v_{k}^{2}}\right)  dx\nonumber\\
&  =4\int_{{B_{{{r_{k,\delta}/r}}_{k}}}\left(  {0}\right)  }\exp\left(
2\phi_{\infty}\right)  \left(  1+\frac{2S_{0}+\phi_{\infty}^{2}+2\phi_{\infty
}}{\gamma_{k}^{2}}\right)  dy+O\left(  {\frac{1}{{\gamma_{k}^{4}}}}\right)
\text{.}\label{3.97}%
\end{align}
By \eqref{128} and a direct calculation, we get
\[
\int_{\mathbb{R}^{2}}\left(  -\Delta S_{0}\right)  dx=4\pi=-\int
_{\mathbb{R}^{2}}4\exp\left(  2\phi_{\infty}\right)  \phi_{\infty}dx\text{,}%
\]
combining this with \eqref{07}, we can obtain
\begin{equation}
\int_{\mathbb{R}^{2}}\exp\left(  2\phi_{\infty}\right)  \left(  2S_{0}%
+\phi_{\infty}^{2}+2\phi_{\infty}\right)  dx=0.\label{add60}%
\end{equation}
Now, using (\ref{127}) again and by a direct calculation, for any $1<\sigma
<2$, we have
\begin{equation}
\int_{%
\mathbb{R}
^{2}\backslash{B_{{{r_{k,\delta}/r}}_{k}}}\left(  {0}\right)  }\exp\left(
2\phi_{\infty}\right)  \left(  1+\phi_{\infty}^{2}+S_{0}\right)  dx=O\left(
\left(  \frac{r_{k}}{r_{k,\delta\text{ }}}\right)  ^{\sigma}\right)
,\label{add69}%
\end{equation}
Combining (\ref{3.97}), (\ref{add60}), (\ref{add69}) and (\ref{127}), we get%
\[
\frac{{4\pi}}{{{E_{k}}}}\int_{B_{r_{k,\delta}}\left(  x_{k}\right)  }%
{v_{k}^{2}}\exp\left(  {v_{k}^{2}}\right)  dy=4\pi+O\left(  {\frac{1}%
{{\gamma_{k}^{4}}}}\right)  \text{.}%
\]
Then the proof is completed.
\end{proof}

\bigskip From the above analysis, we can obtain estimate the exponential term $I_{E}$ and give the

\begin{proof}[Proof of Proposition \ref{129}]
\bigskip By the estimates of $I_{1}$, $I_{2}$, and $I_{3}$, we have%
\begin{align*}
\frac{4\pi}{E_{k}}\int_{\mathbb{R}^{2}}v_{k}^{2}e^{v_{k}^{2}}dx  & =\frac
{4\pi}{E_{k}}\left(  \int_{\mathbb{R}^{2}\backslash B_{L}}+\int_{B_{L}%
\backslash B_{r_{k,\delta}}}+\int_{B_{r_{k,\delta}}}\right)  v_{k}^{2}%
e^{v_{k}^{2}}dx\\
& =I_{1}+I_{2}+I_{3}\\
& =4\pi+O\left(  \frac{1}{\gamma_{k}^{4}}\right)  \text{.}%
\end{align*}
Thus, we complete the proof.

\end{proof}

\end{document}